\documentclass[a4paper,11pt]{amsart}
\usepackage{amssymb}
\usepackage{amscd}
\usepackage{comment}
\usepackage{amsmath,amsthm}
\usepackage[colorlinks=true]{hyperref}
\usepackage{enumerate}
\usepackage{booktabs,multirow}
\usepackage{tikz}
\usepackage{rotating}
\usetikzlibrary{patterns}
\usetikzlibrary{decorations.pathreplacing}
\usetikzlibrary{calc,through}

\allowdisplaybreaks[1]
\numberwithin{figure}{section}
\numberwithin{equation}{section}

\title{Networks bijective to permutations}
\author[K.~Shigechi]{Keiichi~Shigechi}
\email{k1.shigechi AT gmail.com}
\date{\today}

\newcommand\tikzpic[2]{
\raisebox{#1\totalheight}{
\begin{tikzpicture}
#2
\end{tikzpicture}
}}
\newtheorem{theorem}[figure]{Theorem}
\newtheorem{example}[figure]{Example}
\newtheorem{lemma}[figure]{Lemma}
\newtheorem{defn}[figure]{Definition}
\newtheorem{prop}[figure]{Proposition}
\newtheorem{cor}[figure]{Corollary}

\newtheorem{remark}[figure]{Remark}
\begin{document}

\begin{abstract}
We study the set of networks, which consist of sources, sinks and neutral points, bijective 
to the permutations.
The set of directed edges, which characterizes a network, is constructed from a polyomino or a Rothe 
diagram of a permutation through a Dyck tiling on a ribbon.
We introduce a new combinatorial object similar to a tree-like tableau, which we call a forest. 
A forest is shown to give a permutation, and be bijective to a network corresponding to the 
inverse of the permutation.
We show that the poset of networks is a finite graded lattice and admits an $EL$-labeling.
By use of this $EL$-labeling, we show the lattice is supersolvable and compute the M\"obius 
function of an interval of the poset.
\end{abstract}

\maketitle

\section{Introduction}
A network is a graph consisting of vertices and directed edges.
A class of networks are introduced in \cite{Pos06} to study the totally 
non-negative Grassmannian. 
In this paper, we study a special class of networks, which is bijective to the set of permutations.
They have sources and sinks which have only outgoing and incoming edges respectively.
Equivalently, there is no vertex which has outgoing and incoming edges at the same time 
in a network.
We impose one more condition on networks: if vertices $i$ and $k$, and $j$ and $l$ are connected 
by directed edges, then vertices $j$ and $k$ are also connected by a directed edge for $i<j<k<l$.
We establish a bijection between a network and a permutation via a 
set of directed edges in Section \ref{sec:Network}.
Although this bijection depends on the order of directed edges, it is compatible 
with other combinatorial objects, which are polyominoes and forests, appearing in 
Sections \ref{sec:Rothe} and \ref{sec:Poset}.

In section \ref{sec:Rothe}, we consider polyominoes and Rothe diagrams for permutations.
A polyomino $P$ is a diagram consisting of unit squares. 
We consider a polyomino which satisfies some conditions.
To connect a polyomino $P$ with a network, we consider the south-most ribbon in the polyomino $P$.
Here a ribbon is a connected skew shape containing no $2$-by-$2$ rectangles.  
The ribbon $\mathtt{Rib}(P)$ gives a set of directed edges. 
To obtain a set of directed edges form the ribbon, we introduce another combinatorial object 
which is called a Dyck tiling on the ribbon.
A Dyck tiling on a ribbon is a tiling on a ribbon by Dyck tiles which are characterized 
by Dyck paths. It is a special case of Dyck tilings studied in \cite{KeyWil12,KimMesPanWil14,ShiZinJus12}.
The ribbon $\mathtt{Rib}(P)$ also gives another polyomino $P'$ from $P$. Here, the polyomino $P'$ is smaller than $P$.
Then, we obtain another set of directed edges from the new polyomino $P'$ via a ribbon $\mathtt{Rib}(P')$ of $P'$. 
We have a sequence of decreasing polyominoes, and a sequence of ribbons in each step.
Since each ribbon gives the set of directed edges, we have a set $\mathcal{E}(P)$ of directed edges 
by taking a union of the sets of edges constructed from the ribbons. 
We define a network $N(P)$ for $P$ via the set $\mathcal{E}(P)$.
We also have a permutation $\pi:=\pi(P)$ from a polyomino $P$. By a bijection between a permutation and 
a network, we have a network $N(\pi^{-1})$ for an inverse permutation of $\pi$.
We prove that $N(P)=N(\pi^{-1})$.
A Rothe diagram is a visualization of a permutation via unit squares. By generalizing the notion of 
polyominoes, we regard a Rothe diagram as a polyomino with several connected components.
We generalize the results for a polyomino to the case of Rothe diagrams.

A set of networks is regarded as a partially ordered set (poset), and its combinatorial structures
are studied in Section \ref{sec:Poset}.
We first show that the set of networks, which are characterized by the positions of sources and sinks,  
is indeed a finite graded lattice.
By introducing the Whitney numbers of the second kind, we show that the set of networks 
satisfies that the number of elements of even rank is the same as that of odd rank.
Secondly, we study the relations between a forest and a network.
The notion of a forest is close to that of tree-like tableaux studied in \cite{AvaBouNad13} 
and that of $\reflectbox{L}$-diagrams in \cite{Pos06}.
A forest consists of a fixed Young diagram $Y$ and pointed cells in $Y$, and 
it satisfies a condition on pointed cells.
We construct a bijection $\kappa$ between a forest and a network, and interpret the relation $N(P)=N(\pi^{-1})$
in terms of a forest.
Thirdly, we introduce another map $\nu$ from a forest to a permutation. 
The bijection $\kappa$ reflects both the numbers of pointed sells and crossing cells in a forest.
On the other hand, the bijection $\nu$ reflects only the number of pointed cells.

In Section \ref{sec:shell}, we study combinatorial properties of a poset of networks.
A poset of networks is not in general Eulerian. However, it possesses the property that 
the number of elements of even rank is the same as that of odd rank.
By constructing an edge-labeling of a poset which is called $EL$-labeling, we compute the 
M\"obius function of a poset.
For any interval $[x,y]$ in a poset of networks, the M\"obius function $\mu(x,y)$ is 
either $1,-1$ or $0$.
Further, by showing this $EL$-labeling is snelling, we prove that an interval of a poset of 
networks is supersolvable.

\section{Networks with sources and sinks}
\label{sec:Network}
We introduce a notion of networks with sources and sinks.
Let $L$ be a line with $n$ points.
We consider directed edges connecting two points in $L$.
A network consists of points and edges satisfying the following four conditions:
\begin{enumerate}[({A}1)]
\item An edge is directed from the point $i$ to another point $j$ with $i<j$. 
We call the point $i$ a source and $j$ a sink. There exists at most one directed edge from $i$ to $j$.
\item There is no incoming edges and at least one outgoing edges on a source.
\item There is no outgoing edges and at least one incoming edges on a sink.
\item A point is called a neutral point if there is no (outgoing and incoming) edges on the point. 
\end{enumerate}
We denote by $(i,j)$ a directed edge connecting the point $i$ with the point $j$.
A size of a edge $(i,j)$ is defined to be the difference $j-i$.
Suppose we have two edges $(i,k)$ and $(j,l)$.
Two edges are said to be crossing if four points satisfy 
$i<j<k<l$.

We consider the subset of networks satisfying the following property:
\begin{enumerate}[(B1)]
\setcounter{enumi}{0}
\item Suppose a network $N$ contains two directed edges $(i,k)$ and $(j,l)$ with $i<j<k<l$.
Then, $N$ contains the directed edge $(j,k)$.
\end{enumerate}

\begin{defn}
We denote by $\mathcal{N}(n)$ be the set of networks which consist of $n$ points satisfying
the condition (B1).
\end{defn}

For example, the network
\begin{align}
\label{fig:nonad}
\tikzpic{-0.5}{[scale=0.6]
\foreach \x in {1,2,3,4}{
\draw(\x,0)node{$\bullet$}node[anchor=north]{$\x$};}
\draw(2,0)..controls(2,1.5)and(4,1.5)..(4,0);
\draw(1,0)..controls(1,1.5)and(3,1.5)..(3,0);
}
\end{align}
is not admissible.
This network violates the condition (B1) since it does not contain the directed edge $(2,3)$.

\begin{example}
We consider networks with four points, two of which, the points $1$ and $2$, are 
sources and the other two points $3$ and $4$ are sinks.
In fact, we have five such networks as in Figure \ref{fig:n=4}.
\begin{figure}[ht]
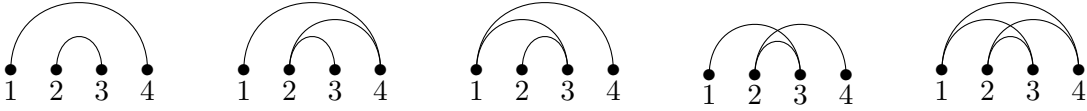

\tikzpic{-0.5}{[scale=0.6]
\foreach \x in {1,2,3,4}{
\draw(\x,0)node{$\bullet$}node[anchor=north]{$\x$};}
\draw(1,0)..controls(1,2)and(4,2)..(4,0);
\draw(2,0)..controls(2,1)and(3,1)..(3,0);
}\quad
\tikzpic{-0.5}{[scale=0.6]
\foreach \x in {1,2,3,4}{
\draw(\x,0)node{$\bullet$}node[anchor=north]{$\x$};}
\draw(1,0)..controls(1,2)and(4,2)..(4,0);
\draw(2,0)..controls(2,1)and(3,1)..(3,0);
\draw(2,0)..controls(2,1.5)and(4,1.5)..(4,0);
}\quad
\tikzpic{-0.5}{[scale=0.6]
\foreach \x in {1,2,3,4}{
\draw(\x,0)node{$\bullet$}node[anchor=north]{$\x$};}
\draw(1,0)..controls(1,2)and(4,2)..(4,0);
\draw(2,0)..controls(2,1)and(3,1)..(3,0);
\draw(1,0)..controls(1,1.5)and(3,1.5)..(3,0);
}\quad
\tikzpic{-0.65}{[scale=0.6]
\foreach \x in {1,2,3,4}{
\draw(\x,0)node{$\bullet$}node[anchor=north]{$\x$};}
\draw(2,0)..controls(2,1.5)and(4,1.5)..(4,0);
\draw(1,0)..controls(1,1.5)and(3,1.5)..(3,0);
\draw(2,0)..controls(2,1)and(3,1)..(3,0);
}\quad
\tikzpic{-0.5}{[scale=0.6]
\foreach \x in {1,2,3,4}{
\draw(\x,0)node{$\bullet$}node[anchor=north]{$\x$};}
\draw(1,0)..controls(1,2)and(4,2)..(4,0);
\draw(2,0)..controls(2,1)and(3,1)..(3,0);
\draw(2,0)..controls(2,1.5)and(4,1.5)..(4,0);
\draw(1,0)..controls(1,1.5)and(3,1.5)..(3,0);
}
\caption{Networks with four points with two sources and two sinks.}
\label{fig:n=4}
\end{figure}
Note that the two right-most networks have a crossing by the edges $(1,3)$ and $(2,4)$, 
and satisfy the condition (B1). We emphasize that a network in Eq. (\ref{fig:nonad}) is not 
admissible. 
\end{example}

We first characterize the networks in $\mathcal{N}(n)$ by permutations.
\begin{theorem}
\label{thrm:cardN}
The cardinality of $\mathcal{N}(n)$ is $n!$, i.e., $|\mathcal{N}_{n}|=n!$.
\end{theorem}

To prove Theorem \ref{thrm:cardN}, we construct a bijection between 
$\mathcal{N}(n)$ and the symmetric group $\mathcal{S}_{n}$ of $n$ elements.
We first construct a map $\sigma:\mathcal{N}(n)\rightarrow\mathcal{S}_{n}$,
and then construct an inverse map from $\mathcal{S}_{n}\rightarrow\mathcal{N}(n)$.
Let $N\in\mathcal{N}(n)$ be a network with $n$ points.
By definition of networks, we have the following property:
\begin{enumerate}[(C1)]
\item There exists no triplet $(i,j,k)$ such that $(i,j)$ and $(j,k)$ are the 
directed edges in $N$.
\end{enumerate}
We introduce a linear order on the edges in $N$.
We first remove the left-most edge $e_1$, which has a minimal size, from $N$.
Then, we remove the left-most edge $e_2$, which has a minimal size, from $N\setminus{e_1}$.
We continue this process until we remove all edges from $N$.
We have a sequence of edges $\{e_{i}\}_{i=1}^{r}$ where $r$ is the number of 
edges in $N$.
Note that this order is well-defined by the property (C1).

For example, the right-most network consisting of four edges in Figure \ref{fig:n=4} 
gives the sequence of the directed edges $\{(2,3),(1,3),(2,4),(1,4)\}$.

Let $\{e_{i}\}_{i=1}^{r}$ be the ordered edges defined as above.
Suppose $\pi:=\pi_1\pi_2\ldots\pi_{n}$ be a permutation in $\mathcal{S}_{n}$.
We define the action of an edge $e:=(i,j)$ of $N$ on $\pi$ as 
the exchange of $\pi_i$ and $\pi_{j}$.
In other words, we regard an edge as a transposition.
We denote by $\pi\xrightarrow{e}\pi'$ the action of the edge $e$ on $\pi$.
Then, we define a permutation $\sigma(N)$ as 
\begin{align*}
\pi^{(0)}\xrightarrow{e_1}\pi^{(1)}\xrightarrow{e_2}\ldots\xrightarrow{e_{r}}\pi^{(r)}=\sigma(N),
\end{align*}
where $\pi^{(0)}$ is an identity, and $\pi^{(i)}\in\mathcal{S}_{n}$.

\begin{remark} 
\label{remark:e1e2}
Suppose two edges $e_1$ and $e_2$ satisfies $e_1\cap e_2=\emptyset$, i.e.,  
$e_1$ and $e_2$ do not have a source or a sink in common.
Then, we can exchange the order of $e_1$ and $e_2$. Namely, we have
\begin{align*}
\tikzpic{-0.5}{
\node(0) at(0,0){$\pi_{0}$};
\node(1)at(1,0.5){$\pi_{1}$};
\node(2)at(1,-0.5){$\pi_2$};
\node(3)at(2,0){$\pi_3$};
\draw[->](0)to node[anchor=south east]{$e_1$}(1);
\draw[->](1)to node[anchor=south west]{$e_2$}(3);
\draw[->](0)to node[anchor=north east]{$e_2$}(2);
\draw[->](2)to node[anchor=north west]{$e_1$}(3);
}
\end{align*}
where $\pi_{i}$, $0\le i\le 3$, are permutations in $\mathcal{S}_{n}$.	
\end{remark}

\begin{example}
Consider the right-most network in Figure \ref{fig:n=4}.
Then, we have 
\begin{align*}
1234\xrightarrow{(2,3)}1324\xrightarrow{(1,3)}2314\xrightarrow{(2,4)}2413\xrightarrow{(1,4)}3412.
\end{align*}
The network corresponds to a permutation $3412$.
\end{example}

Since it is obvious that different networks give distinct permutations, 
the map $\sigma$ is injective. 

We construct an inverse map $\sigma':\mathcal{S}_{n}\rightarrow\mathcal{N}(n)$.
Let $\pi:=\pi_1\ldots\pi_n\in\mathcal{S}_{n}$. 
We recursively construct a graph $N(\pi)$ from $\pi$ as follows.
\begin{enumerate}[(D1)]
\item Suppose $\pi_{n}=n$. Then, the point $n$ in $N(\pi)$ is a neutral point and 
$\pi':=\pi_1\ldots\pi_{n-1}\in\mathcal{S}_{n-1}$ defines a graph on the remaining 
$n-1$ points.
\item Suppose $\pi_{n}\neq n$. Let $j$ be the minimum integer such that 
$\pi_{i}<\pi_{n}$ for $1\le i\le j-1$ and $\pi_{j}>\pi_{n}$.
Then, a graph $N(\pi)$ has an edge $(j,n)$.
We define a new permutation $\pi'$ by $\pi\xrightarrow{(j,n)}\pi'$.
We continue this procedure until we obtain a permutation $\pi'$ whose last element 
is $n$. 
We define a graph as a superposition of the edges $(j,n)$ and the graph 
of $\pi'$. Then, go to (D1).
The algorithm stops when $\pi'$ is the identity permutation.
\end{enumerate}

\begin{example}
Consider the permutation $\pi:=3412$.
Since $\pi_4=2$, we have $j=1$ by (D2). Then, we have $3412\xrightarrow{(1,4)}2413$.
As for $2413$, we have $j=2$ and $2413\xrightarrow{(2,4)}2314$.
The permutation $2314$ has $4$ at the fourth element, we apply (D1) and 
have $(j,n)=(1,3)$ by (D2). Then, we obtain $2314\xrightarrow{(1,3)}1324$.
Finally, we have $2314\xrightarrow{(2,3)}1234$. 
As a summary, we have
\begin{align*}
3412\xrightarrow{(1,4)}2413\xrightarrow{(2,4)}2314\xrightarrow{(1,3)}1324\xrightarrow{(2,3)}1234.
\end{align*}
This sequence gives the set of edges $\{(1,4),(2,4),(1,3),(2,3)\}$.
The network obtained from this set of edges is in the right-most one in Figure \ref{fig:n=4}.
\end{example}

We first show that the inverse map $\sigma'$ is well-defined, that is, 
a network obtained from a permutation $\pi$ is in $\mathcal{N}(n)$.
Especially, we have to show that an obtained network satisfies the condition (B1).
\begin{lemma}
A graph $N(\pi)$ is a network in $\mathcal{N}(n)$.
\end{lemma}
\begin{proof}
Since a graph $N(\pi)$ consists of directed edges and it satisfies (A1) and (A4), 
it is enough to show that $N(\pi)$ satisfies the conditions (A2), (A3) and (B1). 

Suppose that a graph $N(\pi)$ violates a condition (A2) or (A3).
This means that $N(\pi)$ contains two directed edges $(i,j)$ and $(j,k)$ with $i<j<k$.
By construction of $\sigma'$, there exist two permutations $\nu$ and $\nu'$ such hat 
$\nu\xrightarrow{(j,k)}\nu'$.
This implies that $(\nu_{i},\nu_{j},\nu_{k})$ satisfies 
$\nu_{i}<\nu_{k}<\nu_{j}$ and $(\nu'_{i},\nu'_{j},\nu'_{k})=(\nu_{i},\nu_{k},\nu_{j})$.
From (D2), we have a sequence of directed edges $(p,q)$ between $(j,k)$ and $(i,j)$ 
such that $p>j$ if $q=k$, $p<i$ if $q=j$ or $j\le q\le k$.
If we act such directed edges $(p,q)$ on $\nu'$, we have a permutation $\nu''$ 
such that $\nu''_{i}<\nu'_{i}$ and $\nu''_{j}>\nu''_{i}$ by (D2).
The condition $\nu''_{j}>\nu''_{i}$ implies that we have no directed edge $(i,j)$,
which is a contradiction.
The graph $N(\pi)$ does not contain two directed edges $(i,j)$ and $(j,k)$.
Therefore, $N(\pi)$ satisfies the conditions (A2) and (A3). 

We will show that $N(\pi)$ satisfies the condition (B1).
Suppose that $N(\pi)$ has two crossing directed edges $(i,k)$ and $(j,l)$
with $i<j<k<l$, but not the edge $(j,k)$.
There exist two permutations $\nu$ and $\nu'$ such that 
$\nu\xrightarrow{(j,l)}\nu'$. 
By (D2), we have $(\nu'_{i},\nu'_{j},\nu'_{k},\nu'_{l})=(\nu_{i},\nu_{l},\nu_{k},\nu_{j})$
and $\nu_{i}<\nu_{l}<\nu_{j}$. 
Again by (D2), we have a sequence of directed edges $(p,q)$ between $(j,l)$ and $(i,k)$
such that $l\le q\le k$, $p>j$ if $q=l$, and $p<i$ if $q=k$.
This sequence does not contain the directed edge $(p,q)$ such that $p=k$ 
since $N(\pi)$ satisfies (A2) and (A3) as above and $N(\pi)$ contains 
the directed edge $(i,k)$. 
If we act such directed edges $(p,q)$ on $\nu'$ and obtain a permutation $\nu''$,
we have $\nu''_{i}<\nu'_{i}$ and $\nu''_{k}>\nu'_{k}$. 
Since we act the directed edge $(i,k)$ on $\nu''$, we have $\nu''_{i}>\nu''_{k}$ and 
obtain a new permutation $\mu:=(\mu_i,\mu_j,\mu_k,\mu_l)=(\nu''_{k},\nu''_{j},\nu''_{i},\nu''_{l})$.
By combining this condition $\nu''_{i}>\nu''_{k}$ with $\nu_{i}<\nu_{l}<\nu_{j}$, we have 
$\mu_{i}<\mu_{k}<\mu_{j}$. 
The condition $\mu_{i}<\mu_{k}<\mu_{j}$  and the existence of the edge $(j,l)$
imply that  we have the directed edge $(j,k)$ in
$N(\pi)$. 
As a summary, the graph $N(\pi)$ contains the edge $(j,k)$ if it has two crossing 
edges $(j,l)$ and $(i,k)$. 
This completes the proof.
\end{proof}

Since distinct permutations give distinct networks by $\sigma'$, the map $\sigma'$ is injective.

\begin{proof}[Proof of Theorem \ref{thrm:cardN}]
Since $\sigma$ is injective, we have $|\mathcal{N}(n)|\le |\mathcal{S}_{n}|$. 
Similarly, since $\sigma'$ is injective, we have $|\mathcal{N}(n)|\ge |\mathcal{S}_{n}|$.
From these, we have $|\mathcal{N}(n)|=|\mathcal{S}_{n}|=n!$.
\end{proof}

\begin{remark}
Two remarks are in order:
\begin{enumerate}
\item
The maps $\sigma$ and $\sigma'$ are inverse to each other. 
The operation (D2) implies that we take the larger directed edge first when 
we construct a network $N(\pi)$. 
The order of two edges $(i,j)$ and $(k,l)$ with $j<l$ is $(k,l)<(i,j)$ by (D2).
Suppose $e_1$ and $e_2$ be two edges with the same size. Then, the order of $e_1$ and $e_2$
is $e_1<e_2$ in $N(\pi)$ if $e_1$ is right to $e_2$.
The orders of directed edges for $\sigma(N)$ and $\sigma'(\pi)$ 
are reversed to each other, and equivalent up to commuting edges 
(see Remark \ref{remark:e1e2}).
\item The two maps $\sigma$ and $\sigma'$ depend on the order of directed edges.
As explained in (1), the order of edges $\{e_i\}_{i=1}^{r}$ is compatible 
with the order of edges in (D2). 
If we choose a different order, we have a different correspondence between 
a permutation and a network.  
\end{enumerate}
\end{remark}

\section{Rothe diagrams and networks}
\label{sec:Rothe}
\subsection{A Dyck tiling on a ribbon}
A {\it Dyck path} of size $n$ is a lattice path from $(0,0)$ to $(n,n)$ which never 
goes below the line $y=x$. A Dyck path consists of up and right steps. 
We denote by $U$ (resp. $R$) an up (resp. right) step in the path.
For example, we have five Dyck paths of size $3$:
\begin{align*}
URURUR, \quad URUURR, \quad UURRUR, \quad UURURR, \quad UUURRR.
\end{align*}
 
A {\it ribbon} is a connected skew shape which does not contain a $2\times 2$ rectangle.
A {\it Dyck tile} is a ribbon such that the centers of the unit boxes in the ribbon form
a Dyck path. The size of a Dyck tile is defined to be the size of the Dyck path characterizing
the tile.

Since a Dyck tile is a ribbon, one can consider a tiling of a ribbon by use of Dyck tiles.
This tiling is a special case of the cover-inclusive Dyck tiling studied 
in \cite{KeyWil12,KimMesPanWil14,ShiZinJus12}.
A maximal Dyck tiling on a ribbon is a tiling such that each Dyck tile has a maximal 
size.
Figure \ref{fig:DTr} is an example of the maximal Dyck tiling on a ribbon.
We have three Dyck tiles of size zero , two Dyck tiles whose sizes are $1$ and $2$.

\begin{figure}[ht]
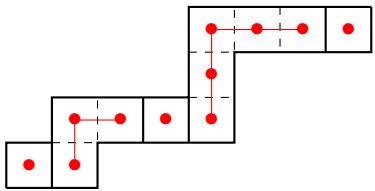

\tikzpic{-0.5}{[scale=0.6]
\draw[thick](0,0)--(0,1)--(1,1)--(1,2)--(4,2)--(4,4)--(8,4);
\draw[thick](0,0)--(2,0)--(2,1)--(5,1)--(5,3)--(8,3)--(8,4);
\draw[thick](1,0)--(1,1)(3,1)--(3,2)(4,1)--(4,2)(7,3)--(7,4);
\draw[dashed](1,1)--(2,1)--(2,2)(4,2)--(5,2)(4,3)--(5,3)--(5,4)(6,4)--(6,3);
\draw[red](0.5,0.5)node{$\bullet$};
\draw[red](1.5,0.5)node{$\bullet$}--(1.5,1.5)node{$\bullet$}--(2.5,1.5)node{$\bullet$};
\draw[red](3.5,1.5)node{$\bullet$};
\draw[red](4.5,1.5)node{$\bullet$}--(4.5,2.5)node{$\bullet$}--(4.5,3.5)node{$\bullet$}
--(5.5,3.5)node{$\bullet$}--(6.5,3.5)node{$\bullet$};
\draw[red](7.5,3.5)node{$\bullet$};
}
\caption{An example of the maximal Dyck tiling on a ribbon. A red line represents a Dyck path which 
characterizes a Dyck tile.}
\label{fig:DTr}
\end{figure}

\subsection{Polyominoes and permutations}
A {\it polyomino} is a diagram consisting of unit squares such that 
two adjacent squares share the same edge.

Let $P$ be a polyomino.
We consider the following three conditions on $P$.
\begin{enumerate}[($\heartsuit$1)]
\item There are no holes in $P$.
\item The heights of the north edges in $P$ are weakly decreasing.
\item The heights of the south edges in $P$ are unimodal, that is, 
we have 
\begin{align*}
s_1\ge s_2\ge\ldots \ge s_{k}\le s_{k+1}\le \ldots \le s_{m},
\end{align*}
where $s_{i}$, $1\le i\le m$, the height of the $i$-th south edge from left.
\end{enumerate}

\begin{defn}
Let $\mathcal{P}$ be the set of polyominoes satisfying the conditions ($\heartsuit$1),
($\heartsuit$2) and  ($\heartsuit$3).
\end{defn}

To connect a polyomino $P$ in $\mathcal{P}$ and a network in $\mathcal{N}(n)$, 
we give a map from a polyomino to a permutation.
We assign positive integers to the east and south edges of $P$ in the following way.
We will obtain a permutation from these integers.
\begin{enumerate}
\item Suppose that an east edge $e$ of $P$ is the east edges of the unit square in the $i$-th row
and $j$-th column. We assign a label $j+x(e)+1$ to this east edge where $x(e)$ is the number of 
east edges which are weakly left and above $e$.  
We write the label right to the east edge. 
If two labels on east edges are in the same column and there is no unit squares of $P$ between them, we move 
the lower label to right by a unit.
\item 
By (1), some south edges in $P$ may have an integer label below them.  
We consider the remaining south edges which have no integer label below them.
Suppose that $E(P)$ be the set of labels assigned to east edges in $P$, 
and a south edge of $P$ is the $i$-th column from left.
We assign the $i$-th smallest element in $\mathbb{Z}_{\ge1}\setminus E(P)$ to this south edge.
We write the label below a south edge.
If two labels on south edges are in the same row, there is an east edge right to the labels, 
and there is no unit squares of $P$ between them, we move the right label downward by a unit.
We do not move further downward if labels are below the lowest south edges of $P$.
\end{enumerate}

We characterize a polyomino $P$ by a permutation which is obtained from the integer labels 
on $P$. 
\begin{defn}
Let $P\in\mathcal{P}$ and $L(P)$ be its labeling.
We define $\alpha:\mathcal{P}\rightarrow\mathcal{S}_{n}$, $P\mapsto\pi$, by 
reading the labels in $L(P)$ from left to right and from top to bottom.
\end{defn}

Figure \ref{fig:polyo} shows an example of polyomino in $\mathcal{P}$ and its labeling.
This polyomino gives the permutation $517\underline{10}264389$.	
\begin{figure}[ht]
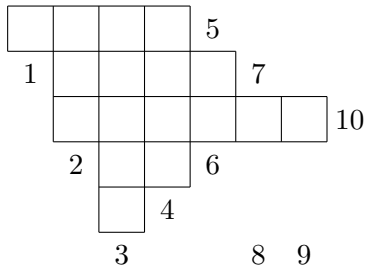

\tikzpic{-0.5}{[scale=0.6]
\draw(0,0)--(4,0)(0,-1)--(5,-1)(1,-2)--(7,-2)(1,-3)--(7,-3)(2,-4)--(4,-4)(2,-5)--(3,-5);
\draw(0,0)--(0,-1)(1,0)--(1,-3)(2,0)--(2,-5)(3,0)--(3,-5)(4,0)--(4,-4)(5,-1)--(5,-3)
(6,-2)--(6,-3)(7,-2)--(7,-3);
\draw(4.5,-0.5)node{$5$}(5.5,-1.5)node{$7$}(7.5,-2.5)node{$10$}(4.5,-3.5)node{$6$}
(3.5,-4.5)node{$4$}(2.5,-5.5)node{$3$}(5.5,-5.5)node{$8$}(6.5,-5.5)node{$9$};
\draw(0.5,-1.5)node{$1$}(1.5,-3.5)node{$2$};
}
\caption{An example of a polyomino in $\mathcal{P}$}
\label{fig:polyo}
\end{figure}

\begin{prop}
\label{prop:alpha}
The map $\alpha$ is well-defined, that is, the word constructed from $L(P)$
is a permutation.
\end{prop}

To prove Proposition \ref{prop:alpha}, we introduce another recursive construction of 
a permutation from a polyomino $P$.
Since $P$ is connected, we enumerate the rows by $1,\ldots,m$ from bottom to top, 
and denote by $r_{i}$ the $i$-th row.
We denote by $l(i):=|r_{i}|$ the number of boxes in the $i$-th row.
Similarly, $l_{L}(i)$ is defined to be the number of boxes in the $i$-th row which 
are left to the left-most box in the $i-1$-th row for $i\ge2$.
We define $l_{L}(1):=l(i)$.
We define a sequence $I(r_{i})$ of integers by 
\begin{align*}
I(r_{i}):=
\begin{cases}
(l(i)+1,1,2,\ldots,l_{L}(i)), & \text{ if } l_{L}(i)>0, \\
(l(i)+1), & \text{ otherwise}.
\end{cases}
\end{align*} 
We will construct a permutation $\pi(P)$ from the collection $\{I(r_i) : 1\le i\le m \}$ of 
the sequences as follows.
\begin{enumerate}
\item Set $i=1$ and $\pi^{(1)}:=I(r_{1})$ in the one-line notation.

\item Let $\pi_{j}$ be the $j$-th element in $\pi^{(i)}$. 
Let $\nu_{j}$ is the $\pi_{j}$-th smallest integer in $\mathbb{Z}_{\ge1}\setminus I(r_{i+1})$.
Then we define $\nu:=(\nu_1,\ldots,\nu_{l})$ where $l$ is the length of $\pi^{(i)}$.
We concatenate $I(r_{i+1})$ and $\nu$ from left to right and denote it by $\nu'$.

\item Suppose that $n=\max \nu'$. If the length $\nu'$ is equal to $n$, we define 
$\pi^{(i+1)}:=\nu'$. Otherwise, define an increasing sequence $\nu''$ consisting of elements 
in $[1,n]\setminus \nu'$.
Then, we define $\pi^{(i+1)}:=\nu'\circ\nu''$, that is, $\pi^{(i+1)}$ is the permutation 
obtained from $\nu'$ by appending $\nu''$ to the right of $\nu'$.
\item Increase $i$ by one, and go to (2). Go to (5) if $i=m$.
\item Define a permutation $\pi(P):=\pi^{(m)}$.
\end{enumerate}

\begin{example}
\label{ex:polyperm}
Consider the polyomino $P$ in Figure \ref{fig:polyo}.
We have five rows in $P$ and the sequences $I(r_{i})$ are given by 
\begin{align*}
I(r_1)=(2,1), \qquad I(r_2)=(3), \qquad I(r_3)=(7,1), \qquad I(r_4)=(5), \qquad I(r_5)=(5,1).
\end{align*}
Then, we have a sequence of permutations 
\begin{align*}
21\xrightarrow{I(r_2)}321\xrightarrow{I(r_3)}7143256\xrightarrow{I(r_4)}
58143267\xrightarrow{I(r_5)}517\underline{10}264389.
\end{align*}
As a summary, we have the permutation $\pi(P)=517\underline{10}264389$.
\end{example}

\begin{remark}
\label{rmk:stand}
The permutation $\pi^{(i)}$ corresponds to the standardization of the 
reading word of $P$ such that it starts from the label $l$ on the east edge in the $i$-th row and 
all the labels are left to or below the label $l$. 
As for Example \ref{ex:polyperm}, we have the following correspondence. 
\begin{align*}
21\leftrightarrow 43 \qquad 321\leftrightarrow643 \qquad 
7143256\leftrightarrow\underline{10}264389\qquad
58143267\leftrightarrow7\underline{10}264389	
\end{align*}
\end{remark}

\begin{proof}[Proof of Proposition \ref{prop:alpha}]
By recursive construction of an integer sequence $\pi(P)$, it is obvious that 
$\pi(P)$ is a permutation.
From Remark \ref{rmk:stand}, 
it is a routine to check that the word constructed from $L(P)$ coincides with the 
permutation $\pi(P)$, which implies that $\alpha(P)=\pi(P)$. Hence, $\alpha(P)$ 
is a permutation.
\end{proof}

\subsection{Polyominoes and networks}
Let $P\in\mathcal{P}$ be a polyomino and $\pi:=\alpha(P)$ be the permutation corresponding 
to $P$.
Recall that we have a bijection between a permutation $\pi^{-1}$ and 
a network $N(\pi^{-1})$.
Therefore, we have a correspondence between the polyomino $P$ and 
the network $N(\pi^{-1})$.
Below, we will construct a network $N'$ by giving a set of 
directed edges from a polyomino and show that $N'$ coincides with $N(\pi^{-1})$.
  
Given a polyomino $P\in\mathcal{P}$ with its labeling, 
let $c_r$ be the unit square in the $y(c_r)$-th row of $P$, whose east edge has the maximal label $n$, 
and $c_{l}$ be the unit square such that the south edge of it is lowest and it has a minimal 
label.  
We consider a ribbon $\mathtt{Rib}(P)$ from $c_{r}$ to $c_{l}$ by taking the unit squares of $P$ 
along the boundary of $P$.
Then, we consider the maximal Dyck tiling of the ribbon $\mathtt{Rib}(P)$.
Let $d_{i}$, $1\le i\le m$, be the Dyck tiles in the maximal Dyck tiling of $\mathtt{Rib}(P)$.
Since a Dyck tile is characterized by a Dyck path, $d_{i}$ has a unique south-most edge.
Let $l_{i}$ be the label of the unique south-most edge in $d_{i}$, and 
$l_{\min}$ be the minimal label among them.
Let $L_{\downarrow}$ be the set of labels left to $l_{\min}$ and strictly below $n$.
Then, we consider a set of directed edges $\mathcal{E}'(P)$:
\begin{align*}
\mathcal{E}'(P):=\{(l_i,n) : 1\le i\le m\}\cup \{(i,n): i\in L_{\downarrow}\}.
\end{align*}

We construct a smaller polyomino $P_1$ from $P$ as follows.
We first delete all unit squares in $\mathtt{Rib}(P)$ from $P$.
Then, we delete a unit square whose south edge has a label in $L_{\downarrow}$. 
Finally, we move the unit squares weakly below the $y(c_r)$-th row rightward by a unit. 
We define $P_1$ by the new polyomino obtained from $P$.
We write 
\begin{align}
\label{eq:seqpolyo}
P=P_{0}\rightarrow P_1\rightarrow P_2\rightarrow \ldots \rightarrow P_{q}\rightarrow P_{q+1}=\emptyset,
\end{align}
if $P_{i+1}$ is obtained from $P_{i}$ by the operation as above. 
Note that we arrive at the empty polyomino since we continue to delete unit squares. 
Then, we define the set $\mathcal{E}(P)$ of directed edges by
\begin{align*}
\mathcal{E}(P):=\bigcup_{i=0}^{q}\mathcal{E}'(P_{i}).
\end{align*}
The network $N(P)$ corresponding to $P$ is obtained from the 
set $\mathcal{E}(P)$ of directed edges.	

\begin{theorem}
\label{thrm:polyo}
Let $P\in\mathcal{P}$, $\pi=\alpha(P)$, and $\mathcal{E}(P)$ be as above.
Then, the network $N(P)$ given by $\mathcal{E}(P)$ coincides with 
the network $N(\pi^{-1})$, that is, we have 
\begin{align*}
N(P)=N(\pi^{-1}).
\end{align*}
\end{theorem}

To prove Theorem \ref{thrm:polyo}, we first show that 
$\mathcal{E}(P)$ gives a network.
We will show that the set $\mathcal{E}(P)$ of directed edges satisfies 
the conditions from (A1) to (A4) in Lemma \ref{lemma:polyoA1A4}, 
and the condition (B1) in Lemma \ref{lemma:polyoB2}.

\begin{lemma}
\label{lemma:polyoA1A4}
Let $\mathcal{E}(P)$ be the set of directed edges as above.
Then, $\mathcal{E}(P)$ satisfies the four conditions from (A1) to (A4).
\end{lemma}
\begin{proof}
It is obvious that the conditions (A1) and (A4) are satisfied.
We will show that $\mathcal{E}(P)$ satisfies the conditions (A2) and (A3).
For this, it is enough to show the following equivalent condition.
\begin{enumerate}
\item[($\star$)] There is no triplet $i<j<k$ such that $(i,j),(j,k)\in\mathcal{E}(P)$.
\end{enumerate}
By definition of $\mathcal{E}'(P_{i})$, a label of a sink comes from 
a label on the east edges in $P$.
Let $n_{i}$ be the maximal label in a polyomino $P_{i}$. 
When we construct $P_{i+1}$ from $P_i$, we delete the ribbon in $P_{i}$ and 
this implies that $n_{i}>n_{i+1}$ for all $i$.
Note that the heights of the south edges are unimodal, 
and that we take a ribbon in $P_{i}$ to consider a maximal Dyck tiling.
These imply that $n_{i+1}$ is the label of an east edge which is above
the label $n_{i}$ and maximal.
Therefore, the construction of $\mathcal{E}(P)$ guarantees that
the labels $n_{j}$, $j>i+1$, do not appear as a label of source. 
This means that $\mathcal{E}(P)$ satisfies the condition ($\star$). 
\end{proof}

\begin{lemma}
\label{lemma:polyoB2}
The set $\mathcal{E}(P)$ of directed edges satisfies the condition (B1).
\end{lemma}
\begin{proof}
Suppose that $(i,k),(j,l)\in\mathcal{E}(P)$ for $i<j<k<l$. 
To show the condition (B1) is equivalent to show $(j,k)\in\mathcal{E}(P)$.
The proof of Lemma \ref{lemma:polyoA1A4} implies that 
$k$ and $l$ are the labels of east edges in $P$, and $k$ is above $l$ in $P$.
Let $l$ and $k$ be the maximal label in $P_{r}$ and $P_{r'}$ for some $r$ and 
$r'$ satisfying $r<r'$ respectively.
Then, by a definition of $\mathcal{E}'(P_r)$, $j$ is a label of a south edge 
in $P_r$ and $i$ is a label of a south edge in $P_{r'}$.
Further, in $P_r$, $i$ is left to and above $j$ since the heights of south edges
are unimodal and $j$ is strictly left to $k$.
We consider a sequence of polyominoes
\begin{align*}
P_{r}\rightarrow P_{r+1}\rightarrow \ldots \rightarrow P_{r'}.
\end{align*}
To obtain $P_{t+1}$ from $P_{t}$ for $r\le t\le r'-1$, we delete the ribbon and some unit squares in $P_{t}$.
By this operation, the labels of the south edges in $P_{t}$ are simply moved upward until 
they become labels of the south edges in $P_{t+1}$.
Since $j$ is a label of a south edge in $P_r$, and $j$ is left to $k$, 
$j$ is again a label of a south edge in $P_{r'}$.
The facts that $j$ is a label of a south edge in $P_r$ and $(j,l)\in\mathcal{E}'(P_r)$ imply 
that $j$ is weakly below and left to the label $k$ in $P_{r'}$.
Further, $(i,k)\in\mathcal{E}'(P_{r'})$ implies that $j$ is a label of the south edge of 
a Dyck tile in $P_{r'}$, which insures that $(j,k)\in\mathcal{E}'(P_{r'})\subseteq\mathcal{E}(P)$.
This completes the proof.
\end{proof}

\begin{proof}[Proof of Theorem \ref{thrm:polyo}]
From Lemmas \ref{lemma:polyoA1A4} and \ref{lemma:polyoB2}, the set $\mathcal{E}(P)$ of 
directed edges gives a network $N(P)$ which satisfies the conditions from (A1) to (A4)
and (B1).
We show that $N(P)=N(\pi^{-1})$.
Let $P_i$ is a polyomino in Eq. (\ref{eq:seqpolyo}), and $\pi_{i}$ be a permutation 
corresponding to $P_{i}$.
In $P_{i}$, the labels of south edges are increasing from left to right.
If the ribbon $\mathtt{Rib}(P_{i})$ contains a Dyck tile $D$ whose size is not zero,
the labels of the south edges in the boundary of $D$ are also increasing. 
Note that these labels are above the label of the south edge of $D$.
Since we obtain a permutation by reading the labels in $P_{i}$ from left to right and 
top to bottom, the existence of the Dyck tile $D$ implies that 
we have a decreasing sequence in $\pi_{i}^{-1}$.
Recall that the map $\sigma'$ on $\pi_{i}^{-1}$ gives the set of directed edges by (D1) and (D2).
Since we consider the ribbon starting from the cell with the maximal label in $P_{i}$,
this corresponds to considering the case (D2) for $\pi_{i}^{-1}$.
Therefore, the case (D2) is compatible with considering the set $\mathcal{E}'(P_{i})$.
As a summary, the set of directed edges obtained from $\pi_{i}^{-1}$ by (D2) is the same 
as $\mathcal{E}'(P_{i})$. This means that $\mathcal{E}(P)$ coincides with the set of 
directed edges for $N(\pi^{-1})$. We have $N(P)=N(\pi^{-1})$.
\end{proof}

\begin{example}
We consider the polyomino $P$ in Figure \ref{fig:polyo}.
We have the following sequence of polyominoes: 
\begin{figure}[ht]
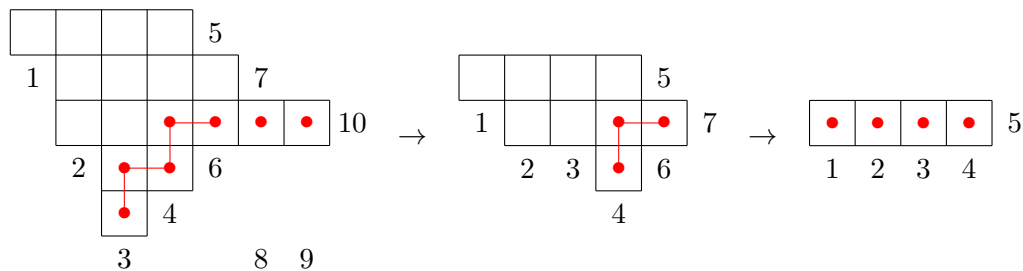

\tikzpic{-0.5}{[scale=0.6]
\draw(0,0)--(4,0)(0,-1)--(5,-1)(1,-2)--(7,-2)(1,-3)--(7,-3)(2,-4)--(4,-4)(2,-5)--(3,-5);
\draw(0,0)--(0,-1)(1,0)--(1,-3)(2,0)--(2,-5)(3,0)--(3,-5)(4,0)--(4,-4)(5,-1)--(5,-3)
(6,-2)--(6,-3)(7,-2)--(7,-3);
\draw(4.5,-0.5)node{$5$}(5.5,-1.5)node{$7$}(7.5,-2.5)node{$10$}(4.5,-3.5)node{$6$}
(3.5,-4.5)node{$4$}(2.5,-5.5)node{$3$}(5.5,-5.5)node{$8$}(6.5,-5.5)node{$9$};
\draw(0.5,-1.5)node{$1$}(1.5,-3.5)node{$2$};
\draw[red](2.5,-4.5)node{$\bullet$}--(2.5,-3.5)node{$\bullet$}--(3.5,-3.5)node{$\bullet$}
--(3.5,-2.5)node{$\bullet$}--(4.5,-2.5)node{$\bullet$};
\draw[red](5.5,-2.5)node{$\bullet$}(6.5,-2.5)node{$\bullet$};
}
$\rightarrow$
\tikzpic{-0.5}{[scale=0.6]
\draw(0,0)--(4,0)(0,-1)--(5,-1)(1,-2)--(5,-2)(3,-3)--(4,-3);
\draw(0,0)--(0,-1)(1,0)--(1,-2)(2,0)--(2,-2)(3,0)--(3,-3)(4,0)--(4,-3)(5,-1)--(5,-2);
\draw(4.5,-0.5)node{$5$}(5.5,-1.5)node{$7$}(4.5,-2.5)node{$6$};
\draw(0.5,-1.5)node{$1$}(1.5,-2.5)node{$2$}(2.5,-2.5)node{$3$}(3.5,-3.5)node{$4$};
\draw[red](3.5,-2.5)node{$\bullet$}--(3.5,-1.5)node{$\bullet$}--(4.5,-1.5)node{$\bullet$};
}
$\rightarrow$
\tikzpic{-0.5}{[scale=0.6]
\draw(0,0)--(4,0)(0,-1)--(4,-1);
\draw(0,0)--(0,-1)(1,0)--(1,-1)(2,0)--(2,-1)(3,0)--(3,-1)(4,0)--(4,-1);
\draw(4.5,-0.5)node{$5$}(0.5,-1.5)node{$1$}(1.5,-1.5)node{$2$}(2.5,-1.5)node{$3$}(3.5,-1.5)node{$4$};
\draw[red](0.5,-0.5)node{$\bullet$}(1.5,-0.5)node{$\bullet$}(2.5,-0.5)node{$\bullet$}(3.5,-0.5)node{$\bullet$};
}
\caption{A sequence of polyominoes.}
\label{fig:red}
\end{figure}
From the left polyomino $P_{0}$, we obtain the set 
$\mathcal{E}'(P_{0})=\{(2,10), (3,10),(8,10),(9,10)\}$.
From the middle polyomino $P_1$, we obtain the set 
$\mathcal{E}'(P_1)=\{(2,7),(3,7),(4,7)\}$.
From the right polyomino $P_2$, we obtain the set 
$\mathcal{E}'(P_2)=\{(1,5),(2,5),(3,5),(4,5)\}$.  
The network $N(P)$ is characterized by the set of directed edges
$\mathcal{E}(P)=\mathcal{E}'(P_{0})\cup\mathcal{E}'(P_{1})\cup\mathcal{E}'(P_{2})$.

The polyomino $P$ gives the permutation $\pi=517\underline{10}264389$.
Then, the inverse permutation is $\pi^{-1}=25871639\underline{10}4$.
It is easy to see that the network $N(\pi^{-1})$ gives the same set of 
directed edges as $\mathcal{E}(P)$. 
\end{example}

\subsection{Rothe diagrams and networks}
Let $\pi$ be a permutation in $\mathcal{S}_{n}$.
The diagram, called the {\it Rothe diagram}, is defined as the 
set of unit boxes:
\begin{align*}
D(\pi):=\{(i,j)| 1\le i,j\le n, \pi(i)>j, \pi^{-1}(j)>i\}.
\end{align*}
Note that the Rothe diagram $D(\pi^{-1})$ is the transposed 
diagram of $D(\pi)$. 
Here, transposition means that we exchange rows and columns in $D(\pi)$.
When $\pi=\pi_1\ldots\pi_{n}$, we write an integer $\pi_{i}$ in the $i$-th 
row and $\pi_{i}$-th column in $D(\pi)$. 
For example, the Rothe diagram for $263514$ is given in Figure \ref{fig:DTRothe}.
\begin{figure}[ht]
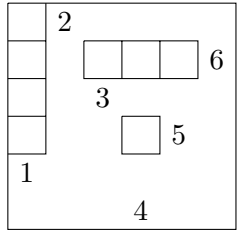

\tikzpic{-0.5}{[scale=0.5]
\draw(0,0)--(6,0)--(6,-6)--(0,-6)--(0,0);
\draw(1,0)--(1,-4)--(0,-4)(1,-1)--(0,-1)(1,-2)--(0,-2)(1,-3)--(0,-3);
\draw(2,-1)--(5,-1)--(5,-2)--(2,-2)--(2,-1)(3,-1)--(3,-2)(4,-1)--(4,-2);
\draw(3,-3)--(4,-3)--(4,-4)--(3,-4)--(3,-3);
\draw(1.5,-0.5)node{$2$}(5.5,-1.5)node{$6$}(2.5,-2.5)node{$3$}(4.5,-3.5)node{$5$}
(0.5,-4.5)node{$1$}(3.5,-5.5)node{$4$};
}
\caption{The Rothe diagram for $263514$.}
\label{fig:DTRothe}
\end{figure}
In general, a Rothe diagram consists of several connected components.
We glue these connected components into a larger connected component 
by keeping the connectivity of squares. 
Here, connectivity of two squares means that one square is below or right to 
another square.
Even if we glue components, we may have several larger connected components.

For example, the Rothe diagram $D(263514)$ in Figure \ref{fig:DTRothe} has 
three connected components. 
Let $C_1$, $C_{2}$ and $C_{3}$ be the connected components consisting of 
four, three and one squares respectively.
By moving $C_2$ and $C_{3}$ rightward by one unit, we can glue $C_1$ and $C_{2}$
into a larger component $C_{1\cup2}$. However, we cannot glue $C_{3}$ 
and $C_{1\cup2}$ since if we move horizontally or vertically $C_{3}$ by one unit, 
we have to change the connectivity of squares.
As a consequence, we have two components $C_{1\cup2}$ and $C_{3}$. 
 
\begin{defn}
Let $\pi\in\mathcal{S}_{n}$ and $D(\pi)$ be the Rothe diagram of $\pi$.
We call the set of maximal polyominoes, which are constructed from $D(\pi)$ by gluing the connected components, 
a polyomino for $\pi$. A connected component in the polyomino is called a component.
\end{defn}

By definition, it is clear that the polyomino for $263514$ has two components $C_{1\cup2}$ and $C_{3}$, and 
these components are maximal.
When we merge two components into a larger component, we move the labels in the Rothe diagram in such a
way that they are compatible with the connectivity between the labels and unit squares. 

Let $\pi\in\mathcal{S}_{n}$ and $P(\pi)$ be the polyomino for $\pi$.
Let $c_{r}$ be the unit square in $P(\pi)$ such that the label right to 
$c_{r}$ is $n$.
We define $c_{l}$ to be the unit square in $P(\pi)$ such that it is the left-most square 
in the lowest row. 
As in the case of polyominoes in $\mathcal{P}$, we consider the ribbon from $c_{r}$ 
to $c_{l}$ by taking unit squares along the boundary squares in $P(\pi)$.
Here, a ribbon may have several components and be no longer a skew shape, 
we focus on only the connectivity of squares in $P(\pi)$.
We denote by $\mathtt{Rib}(\pi)$ the ribbon from $c_r$ to $c_{l}$.
We consider the maximal Dyck tiling on $\mathtt{Rib}(\pi)$.
We emphasize that we look at only the connectivity of squares in a Dyck tile,
that is, a Dyck tile may consist of squares in different several components.
Let $d_i$, $1\le i\le m$, be Dyck tiles of $\mathtt{Rib}(\pi)$.
We denote by $l_i$, $1\le i\le m$, the label of the south edge of $d_i$ in $P(\pi)$,
and by $l_{\min}$ the minimal integer in $\{l_i:1\le i\le m\}$, where $m$ is the number of 
Dyck tiles in $P(\pi)$.
The set $L^{\downarrow}:=\{l^{\downarrow}_1,\ldots,l^{\downarrow}_{s}\}$ of integers is defined to be 
the integers in one-line notation of $\pi$ such that they are a maximal decreasing sequence, 
and they are left to $l_{\min}$ and right to $n$, that is, 
$L^{\downarrow}$ satisfies 
\begin{enumerate}
\item $l^{\downarrow}_1:=l_{\min}$, 
\item $l^{\downarrow}_{i+1}<l^{\downarrow}_{i}$ and 
$l^{\downarrow}_{i+1}$ is left to $l^{\downarrow}_{i}$. Further, the integers between $l_{i+1}$ and $l_{i}$
are larger than $l_{i}$,
\item $s$ is maximal and $l^{\downarrow}_{s}$ is right to $n$ in $\pi$. 
\end{enumerate}
For example, if $\pi=81362475$ and $l_{\min}=5$, then 
we have $L^{\downarrow}=\{5,4,2,1\}$.

We define the set $\mathcal{E}'(\pi)$ of directed edges by 
\begin{align*}
\mathcal{E}'(\pi):=
\{(l_i,n) : 1\le i\le m\}\cup
\{(i,n) : i\in L^{\downarrow} \}.
\end{align*}
Define the set $I(\pi)$ of labels by
\begin{align}
\label{eq:setIpi}
I(\pi):=\{l_i:1\le i\le m\}\cup L^{\downarrow}\cup\{n\}.
\end{align}
Then, we construct a new permutation $\pi_1$ in one-line notation from $I(\pi)$ 
in such a way that we keep the positions of integers $\{1,2,\ldots,n\}\setminus I(\pi)$
as it is and we reorder the integers in $I(\pi)$ in an increasing order from left to right.
For example, when $\pi=81362475$ and $I(\pi)=\{1,2,4,5,8\}$, 
we have $\pi_1=12364578$. 
We write $\pi\xrightarrow{I(\pi)}\pi_1$, or simply $\pi\rightarrow\pi_1$.
Then, we have a sequence of permutations
\begin{align}
\label{eq:seqpi}
\pi=\pi_0\rightarrow \pi_1\rightarrow \ldots\rightarrow\pi_{t}\rightarrow \pi_{t+1}=\mathrm{id}.
\end{align}
Define the set $\mathcal{E}(\pi)$ of directed edges by
\begin{align*}
\mathcal{E}(\pi):=\bigcup_{i=0}^{t}\mathcal{E}'(\pi_{i}).
\end{align*}

\begin{theorem}
\label{thrm:Rothe}
Let $\pi\in\mathcal{S}_{n}$ and $\mathcal{E}(\pi)$ be the set of directed edges as above.
Then, $\mathcal{E}(\pi)$ coincides with the network $N(\pi^{-1})$. 
\end{theorem}

To prove Theorem \ref{thrm:Rothe}, we introduce a procedure to obtain a polyomino $P_{i+1}$
for $\pi_{i+1}$ from $P_{i}$ where $\pi_{i}$ is a permutation in Eq. (\ref{eq:seqpi}).
We regard $I(\pi_{i}):=(k_1,\ldots,k_{m})$ as an increasing integer sequence. 
We define $n:=\max I(\pi_{i})$, that is, $n=k_m$.
The integers $\{k_{j}:1\le j\le m-1 \}$ in $I(\pi_{i})$ are labels on south edges of $P_{i}$, 
and the integer $k_{m}$ is a label on an east edge of $P_{i}$.
Suppose the integer $k_{j}$ is in the $r_{j}$-th row from top.
To obtain a polyomino $P_{i+1}$, we focus on the positions of labels.
We move the integer $k_{j}$ upward or downward such that it is in the $r_{j-1}$-th row for $j>2$, and 
the integer $k_{1}$ is moved upward to the $r_{m}$-th row.
Then, we obtain a polyomino with labels. 
By keeping the components of the polyomino, we may move the components to give a compatible 
polyomino with a permutation.

\begin{example}
We consider the polyomino for $\pi=263514$ and $I(\pi)=(1,4,6)$.
The label $6$ is in the second row, and the labels $1$ and $4$ are in the fifth row in the polyomino.
By this, we transform the polyomino as follows.
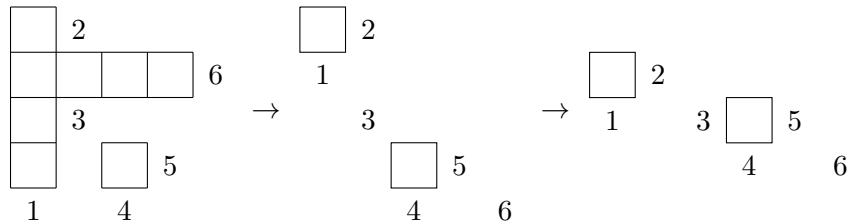
\begin{align*}
\tikzpic{-0.5}{[scale=0.6]
\draw(0,0)--(1,0)(0,-1)--(4,-1)(0,-2)--(4,-2)(0,-3)--(1,-3)(2,-3)--(3,-3)(0,-4)--(1,-4)(2,-4)--(3,-4);
\draw(0,0)--(0,-4)(1,0)--(1,-4)(2,-1)--(2,-2)(2,-3)--(2,-4)(3,-1)--(3,-2)(3,-3)--(3,-4)(4,-1)--(4,-2);
\draw(1.5,-0.5)node{$2$}(4.5,-1.5)node{$6$}(1.5,-2.5)node{$3$}(3.5,-3.5)node{$5$}
(2.5,-4.5)node{$4$}(0.5,-4.5)node{$1$};
}
\rightarrow
\tikzpic{-0.5}{[scale=0.6]
\draw(0,0)--(1,0)--(1,-1)--(0,-1)--(0,0);
\draw(2,-3)--(3,-3)--(3,-4)--(2,-4)--(2,-3);
\draw(1.5,-0.5)node{$2$}(0.5,-1.5)node{$1$}(1.5,-2.5)node{$3$}(3.5,-3.5)node{$5$}(2.5,-4.5)node{$4$}(4.5,-4.5)node{$6$};
}
\rightarrow
\tikzpic{-0.5}{[scale=0.6]
\draw(0,0)--(1,0)--(1,-1)--(0,-1)--(0,0);
\draw(3,-1)--(4,-1)--(4,-2)--(3,-2)--(3,-1);
\draw(1.5,-0.5)node{$2$}(0.5,-1.5)node{$1$}(2.5,-1.5)node{$3$}(4.5,-1.5)node{$5$}(3.5,-2.5)node{$4$}(5.5,-2.5)node{$6$};
}
\end{align*}
Note that the middle polyomino is not compatible with the permutation $213546$, but 
the right one is compatible.	
\end{example}

\begin{proof}[Proof of Theorem \ref{thrm:Rothe}]
By applying the same argument as in Lemmas \ref{lemma:polyoA1A4} and \ref{lemma:polyoB2} to 
the set $\mathcal{E}(\pi)$ of directed edges, one can show that $\mathcal{E}(\pi)$ 
gives a network satisfying from (A1) to (A4) and (B1). 
We also apply the same argument as in the proof of Theorem \ref{thrm:polyo} to $\mathcal{E}(\pi)$.
Then, it is a routine to show that the set $\mathcal{E}(\pi)$ gives the same network
as $N(\pi^{-1})$.
\end{proof}

\begin{example}
Let $\pi=3164752$ and $\pi^{-1}=2714635$.
The Rothe diagram $D(\pi^{-1})$ has two components.
\begin{figure}[ht]
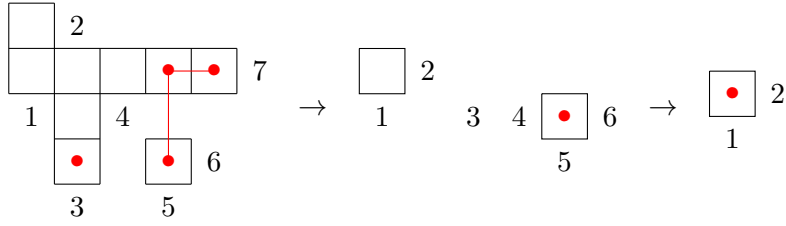

\tikzpic{-0.5}{[scale=0.6]
\draw(0,0)--(1,0)(0,-1)--(5,-1)(0,-2)--(5,-2)(1,-3)--(2,-3)(3,-3)--(4,-3)
(1,-4)--(2,-4)(3,-4)--(4,-4);
\draw(0,0)--(0,-2)(1,0)--(1,-4)(2,-1)--(2,-4)(3,-1)--(3,-2)(3,-3)--(3,-4)
(4,-1)--(4,-2)(4,-3)--(4,-4)(5,-1)--(5,-2);
\draw(1.5,-0.5)node{$2$}(5.5,-1.5)node{$7$}(0.5,-2.5)node{$1$}(2.5,-2.5)node{$4$}
(1.5,-4.5)node{$3$}(3.5,-4.5)node{$5$}(4.5,-3.5)node{$6$};
\draw[red](4.5,-1.5)node{$\bullet$}--(3.5,-1.5)node{$\bullet$}--(3.5,-3.5)node{$\bullet$}
(1.5,-3.5)node{$\bullet$};
}
$\rightarrow$
\tikzpic{-0.5}{[scale=0.6]
\draw(0,0)--(1,0)--(1,-1)--(0,-1)--(0,0);
\draw(1.5,-0.5)node{$2$}(0.5,-1.5)node{$1$}(2.5,-1.5)node{$3$}(3.5,-1.5)node{$4$};
\draw(4,-1)--(5,-1)--(5,-2)--(4,-2)--(4,-1);
\draw(5.5,-1.5)node{$6$}(4.5,-2.5)node{$5$};
\draw[red](4.5,-1.5)node{$\bullet$};
}
$\rightarrow$
\tikzpic{-0.5}{[scale=0.6]
\draw(0,0)--(1,0)--(1,-1)--(0,-1)--(0,0);
\draw(1.5,-0.5)node{$2$}(0.5,-1.5)node{$1$};
\draw[red](0.5,-0.5)node{$\bullet$};
}
\caption{A sequence of permutations for $2714635$.}
\label{fig:RNet}
\end{figure}
The left polyomino gives the set of directed edges $\{(1,7),(3,7),(5,7)\}$.
The middle and right polyominoes give the set $\{(5,6)\}$ and $\{(1,2)\}$.
From these, the set $\mathcal{E}(\pi^{-1})$ of edges is given by 
\begin{align*}
\{(1,2),(1,7),(3,7),(5,6),(5,7)\}.
\end{align*}
It is easy to see that the network $N(\pi)$ has also the same 
set of directed edges.
\end{example}

\section{A poset of networks}
\label{sec:Poset}
\subsection{Basic properties of a poset of networks}
Let $\epsilon:=\epsilon_1\ldots\epsilon_{n}\in\{1,0,-1\}^{n}$ be a sequence such that 
$\epsilon_{j}=1$ if $\epsilon_{i}=0$ for $1\le i\le j-1$.
In other words, a sequence $\epsilon$ starts from $1$ if we ignore the zeroes.
A sequence $\epsilon$ specifies the sources, sinks and neutral points on the line 
with $n$ points. The point $i$ is a source or a neutral point if $\epsilon=1$, 
a sink or a neutral point if $\epsilon=-1$, and a neutral point if $\epsilon=0$.
Let $\mathcal{N}(n;\epsilon)\subset\mathcal{N}(n)$ be the set of networks such that 
the positions of sources, sinks and neutral points are characterized by $\epsilon$.

Let $N\in\mathcal{N}(n)$ be a network with $m$ directed edges.
Then, we define a function $\rho:\mathcal{N}(n)\rightarrow\mathbb{Z}_{\ge0}$ 
by $\rho(N)=m$.
We have a graded set by this function.
Later, we see that the function $\rho$ is the rank function of the poset of networks.

Let $x,y\in\mathcal{N}(n;\epsilon)$ be networks, and denote by $\mathcal{E}(x)$
be the set of directed edges in $x$.

\begin{defn}
A network $y$ covers $x$ if and only if $\rho(y)=\rho(x)+1$ and $\mathcal{E}(x)\subset\mathcal{E}(y)$.
When $y$ covers $x$, we write $x\lessdot y$.
We write $x\le y$ if we have a sequence of networks $x=z_0\lessdot z_1\lessdot\ldots\lessdot z_{r}=y$
with $r\ge0$.
\end{defn}

Let $\mathtt{S}_{\uparrow}(\epsilon)$ (resp. $\mathtt{S}_{\downarrow}(\epsilon)$) 
be the set of indices $i$ such that $\epsilon_{i}=1$ (resp. $\epsilon_{i}=-1$).
We denote by $N_{\mathrm{max}}(\epsilon)$ a unique network which has the maximal number of 
edges in $\mathcal{N}(n;\epsilon)$.
The set of directed edges in $N_{\mathrm{max}}(\epsilon)$ is given by
\begin{align*}
\mathcal{E}(N_{\mathrm{max}}(\epsilon))
=\{(i,j) |1\le i<j\le n, i\in\mathtt{S}_{\uparrow}(\epsilon), j\in\mathtt{S}_{\downarrow}(\epsilon) \}.
\end{align*}
By construction, the network $N_{\max}(\epsilon)$ is unique.

\begin{defn}
We define a graded partially ordered post (poset) $\mathcal{P}(n;\epsilon)$ by 
$\mathcal{P}(n;\epsilon):=(\mathcal{N}(n;\epsilon),\le)$.
\end{defn}

The poset $\mathcal{P}(n;\epsilon)$ has a minimum element $\hat{0}$ and a
maximum element $\hat{1}$. 
The element $\hat{0}$ is the network without edges, i.e., the network 
corresponding to the identity permutation.
The element $\hat{1}$ is given by $N_{\mathrm{max}}(\epsilon)$.

By definition of the covering relation, note that the function $\rho$ is the rank 
function of $\mathcal{P}(n;\epsilon)$.

\begin{figure}[ht]
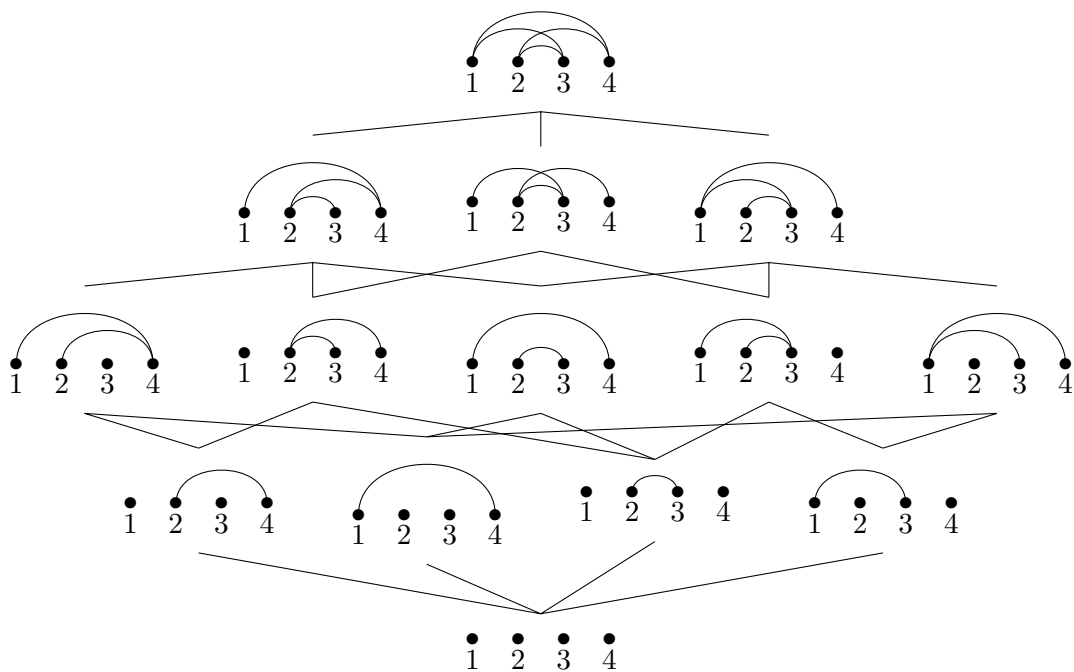

\tikzpic{-0.5}{
\node (3412) at (0,0){
\tikzpic{-0.5}{[scale=0.6]
\foreach \x in{1,2,3,4}{
\draw(\x,0)node{$\bullet$}node[anchor=north]{$\x$};
}
\draw(1,0)..controls(1,1.5)and(4,1.5)..(4,0);	
\draw(1,0)..controls(1,1)and(3,1)..(3,0);
\draw(2,0)..controls(2,1)and(4,1)..(4,0);
\draw(2,0)..controls(2,0.5)and(3,0.5)..(3,0);
}
};
\node (3421) at (-3,-2){
\tikzpic{-0.5}{[scale=0.6]
\foreach \x in{1,2,3,4}{
\draw(\x,0)node{$\bullet$}node[anchor=north]{$\x$};
}
\draw(1,0)..controls(1,1.5)and(4,1.5)..(4,0);	
\draw(2,0)..controls(2,1)and(4,1)..(4,0);
\draw(2,0)..controls(2,0.5)and(3,0.5)..(3,0);
}
};
\node (2413) at (0,-2){
\tikzpic{-0.5}{[scale=0.6]
\foreach \x in{1,2,3,4}{
\draw(\x,0)node{$\bullet$}node[anchor=north]{$\x$};
}
\draw(1,0)..controls(1,1)and(3,1)..(3,0);
\draw(2,0)..controls(2,1)and(4,1)..(4,0);
\draw(2,0)..controls(2,0.5)and(3,0.5)..(3,0);
}
};
\node (4312) at (3,-2){
\tikzpic{-0.5}{[scale=0.6]
\foreach \x in{1,2,3,4}{
\draw(\x,0)node{$\bullet$}node[anchor=north]{$\x$};
}
\draw(1,0)..controls(1,1.5)and(4,1.5)..(4,0);	
\draw(1,0)..controls(1,1)and(3,1)..(3,0);
\draw(2,0)..controls(2,0.5)and(3,0.5)..(3,0);
}
};
\node (1423) at (-3,-4){
\tikzpic{-0.5}{[scale=0.6]
\foreach \x in{1,2,3,4}{
\draw(\x,0)node{$\bullet$}node[anchor=north]{$\x$};
}
\draw(2,0)..controls(2,1)and(4,1)..(4,0);
\draw(2,0)..controls(2,0.5)and(3,0.5)..(3,0);
}
};
\node (2431) at (-6,-4){
\tikzpic{-0.5}{[scale=0.6]
\foreach \x in{1,2,3,4}{
\draw(\x,0)node{$\bullet$}node[anchor=north]{$\x$};
}
\draw(1,0)..controls(1,1.5)and(4,1.5)..(4,0);	
\draw(2,0)..controls(2,1)and(4,1)..(4,0);
}
};
\node (4321) at (0,-4){
\tikzpic{-0.5}{[scale=0.6]
\foreach \x in{1,2,3,4}{
\draw(\x,0)node{$\bullet$}node[anchor=north]{$\x$};
}
\draw(1,0)..controls(1,1.5)and(4,1.5)..(4,0);	
\draw(2,0)..controls(2,0.5)and(3,0.5)..(3,0);
}
};
\node (4213) at (6,-4){
\tikzpic{-0.5}{[scale=0.6]
\foreach \x in{1,2,3,4}{
\draw(\x,0)node{$\bullet$}node[anchor=north]{$\x$};
}
\draw(1,0)..controls(1,1.5)and(4,1.5)..(4,0);	
\draw(1,0)..controls(1,1)and(3,1)..(3,0);
}
};
\node (2314) at (3,-4){
\tikzpic{-0.5}{[scale=0.6]
\foreach \x in{1,2,3,4}{
\draw(\x,0)node{$\bullet$}node[anchor=north]{$\x$};
}
\draw(1,0)..controls(1,1)and(3,1)..(3,0);
\draw(2,0)..controls(2,0.5)and(3,0.5)..(3,0);
}
};
\node (3214) at (4.5,-6){
\tikzpic{-0.5}{[scale=0.6]
\foreach \x in{1,2,3,4}{
\draw(\x,0)node{$\bullet$}node[anchor=north]{$\x$};
}
\draw(1,0)..controls(1,1)and(3,1)..(3,0);
}
};
\node (1324) at (1.5,-6){
\tikzpic{-0.5}{[scale=0.6]
\foreach \x in{1,2,3,4}{
\draw(\x,0)node{$\bullet$}node[anchor=north]{$\x$};
}
\draw(2,0)..controls(2,0.5)and(3,0.5)..(3,0);
}
};
\node (4231) at (-1.5,-6){
\tikzpic{-0.5}{[scale=0.6]
\foreach \x in{1,2,3,4}{
\draw(\x,0)node{$\bullet$}node[anchor=north]{$\x$};
}
\draw(1,0)..controls(1,1.5)and(4,1.5)..(4,0);	
}
};
\node (1432) at (-4.5,-6){
\tikzpic{-0.5}{[scale=0.6]
\foreach \x in{1,2,3,4}{
\draw(\x,0)node{$\bullet$}node[anchor=north]{$\x$};
}
\draw(2,0)..controls(2,1)and(4,1)..(4,0);
}
};
\node (1234) at (0,-8){
\tikzpic{-0.5}{[scale=0.6]
\foreach \x in{1,2,3,4}{
\draw(\x,0)node{$\bullet$}node[anchor=north]{$\x$};
}
}
};
\draw(3412.south)--(3421.north)(3412.south)--(2413.north)(3412.south)--(4312.north);
\draw(3421.south)--(1423.north)(3421.south)--(4321.north)(3421.south)--(2431.north);
\draw(2413.south)--(1423.north)(2413.south)--(2314.north);
\draw(4312.south)--(2314.north)(4312.south)--(4321.north)(4312.south)--(4213.north);
\draw(2431.south)--(4231.north)(2431.south)--(1432.north);
\draw(1423.south)--(1432.north)(1423.south)--(1324.north);
\draw(4321.south)--(4231.north)(4321.south)--(1324.north);
\draw(4213.south)--(4231.north)(4213.south)--(3214.north);
\draw(2314.south)--(3214.north)(2314.south)--(1324.north);
\draw(3214.south)--(1234.north)(1324.south)--(1234.north)(4231.south)--(1234.north)
(1432.south)--(1234.north);
}
\caption{A poset $\mathcal{P}(4;\epsilon)$ with $\epsilon=(1,1,-1,-1)$.}
\label{fig:Eulerian}
\end{figure}

An example of the poset $\mathcal{P}(4;\epsilon)$ with $\epsilon=(1,1,-1,-1)$ 
is shown in Figure \ref{fig:Eulerian}.
Note that the network in Eq. (\ref{fig:nonad}) does not appear in 
the poset.

We briefly recall the definition of a Eulerian poset 
following \cite{Sta94,Sta97b1}.
Let $P$ be a finite graded poset of rank $n+1$ with $\hat{0}$ and $\hat{1}$.
Let $\mu$ be the M\"obius function of a poset $P$, and $\rho$ the rank function.
Thus we have $\rho(\hat{0})=0$ and $\rho(\hat{1})=n+1$.
Given two elements $x\le y$ in $P$, we write $\rho(x,y):=\rho(y)-\rho(x)$.
The function $\rho(x,y)$ is the rank of the interval $[x,y]$.

\begin{defn}
\label{defn:Eulerian}
A poset $P$ is {\it Eulerian} if $\mu(x,y)=(-1)^{\rho(x,y)}$ 
for all $x\le y$ in $P$. 
\end{defn}

Let (E1) be the following statement for a poset $P$:
\begin{enumerate}[(E1)]
\item 
The number of elements of even rank is 
equal to that of odd rank in $P$. 
\end{enumerate}
Definition \ref{defn:Eulerian} implies that the statement (E1) holds for 
every interval of rank at least one.
In terms of the rank function, Definition \ref{defn:Eulerian} means that we have 
\begin{align*}
\sum_{z\in[x,y]}(-1)^{\rho(z)}=0,
\end{align*}
if $x<y$ in $P$.

To show that the poset $\mathcal{P}(n;\epsilon)$ is a lattice, 
we define the join (or least upper bound) $x\vee y$ and the meet (or greatest lower bound)$ x\wedge y$ for 
two elements $x,y\in\mathcal{P}(n;\epsilon)$.

Recall that an element $x\in\mathcal{P}(n;\epsilon)$ can be characterized by 
the set of directed edges.
This means that we have an obvious bijection between a network and the set of 
directed edges.
Recall that $\mathcal{E}(x)$ is the set of directed edges in the network $x$.
Then, we define the set of directed edges for the meet $x\wedge y$ by 
\begin{align*}
\mathcal{E}(x\wedge y):=\mathcal{E}(x)\cap\mathcal{E}(y).
\end{align*}
In the case of the join, we define 
\begin{align*}
\mathcal{E}(x\vee y):=\mathcal{E}(x)\cup\mathcal{E}(y)\cup \mathcal{E}^{\times}(x,y).
\end{align*}
The set $\mathcal{E}^{\times}(x,y)$ is defined as follows.
Suppose that the two directed edges $(i,k)$ and $(j,l)$ in $\mathcal{E}(x)\cup\mathcal{E}(y)$ 
are crossing where $i<j<k<l$.
Then, the set $\mathcal{E}^{\times}(x,y)$ is 
\begin{align*}
\mathcal{E}^{\times}(x,y):=
\{(j,k)| (i,k),(j,l)\in\mathcal{E}(x)\cup\mathcal{E}(y)\}.
\end{align*}
Then, it is a routine to check that 
$x\wedge y\le x$, $x\wedge y\le y$, and if $z\le x$ and $z\le y$, then $z\le x\wedge y$ for any $z$.
Similarly, we have $x\le x\vee y$, $y\le x\vee y$, and if $x\le z$ and $y\le z$, then $x\vee y\le z$
for any $z$.

\begin{example}
Suppose $\mathcal{E}(x)=\{(1,3)\}$ and $\mathcal{E}(y)=\{(2,4)\}$.
Then, we have $\mathcal{E}(x\wedge y)=\emptyset$.
For the join, we have $\mathcal{E}(x\vee y)=\{(1,3),(2,4),(2,3)\}$.
Note that the network in Eq. (\ref{fig:nonad}) whose directed 
edges are $\{(1,3),(2,4)\}$ is not admissible. 
\end{example}

The next proposition is a direct consequence of the observations above.
\begin{prop}
The poset $\mathcal{P}(n;\epsilon)$ is a finite graded lattice.
\end{prop}

\begin{prop}
\label{prop:Boolean}
The graded poset $\mathcal{P}(n;\epsilon)$ is a Boolean lattice if $N_{\mathrm{max}}(\epsilon)$
has no crossing edges. Hence, it is Eulerian.
\end{prop}
\begin{proof}
Since $N_{\mathrm{max}}(\epsilon)$ has no crossing edges, an element $\mathcal{P}(n;\epsilon)$ 
can be uniquely expressed as the join of atoms. Here, an atom is an element in $\mathcal{P}(n;\epsilon)$
such that it contains only one directed edge.
This implies that $P(n;\epsilon)$ is a Boolean lattice.
It is easy to see that a Boolean lattice is Eulerian, 
which completes the proof.	
\end{proof}

For general $\epsilon$, the poset $\mathcal{P}(n;\epsilon)$ is not Eulerian.
This can be easily seen when $N_{\mathrm{max}}(\epsilon)$ has a crossing,
the poset $\mathcal{P}(n;\epsilon)$ contains the subposet as in Figure \ref{fig:LE}.
It is obvious that this subposet is not Eulerian.

However, $\mathcal{P}(n;\epsilon)$ has the following property.

\begin{prop}
\label{prop:evenodd}
In $\mathcal{P}(n;\epsilon)$, the number of elements of even rank is the same as 
that of odd rank.
\end{prop}

To prove Proposition \ref{prop:evenodd}, we introduce the notion of 
Whitney numbers.
We consider the {\it Whitney numbers $W_{r}(\epsilon)$ of 
the second kind} defined by 
\begin{align*}
W_{r}(\epsilon):=\#\{N\in\mathcal{P}(n;\epsilon)| \rho(N)=r \}.
\end{align*}
The number $W_{r}$ is the number of elements of $\mathcal{P}(n;\epsilon)$ of rank $r$. 
Then, we define the ordinary generating function by 
\begin{align*}
\mathcal{W}(\epsilon)=\sum_{r=0}^{n+1}q^{r} W_{r}(\epsilon).
\end{align*}
Let $\epsilon'$ be a sequence of $1$, $0$ and $-1$ obtained from $\epsilon$
by deleting several zeroes.
Since a zero corresponds to a neutral point in a network, it is obvious 
that $\mathcal{W}(\epsilon')=\mathcal{W}(\epsilon)$.
Thus, we assume that $\epsilon:=\epsilon_1\ldots\epsilon_{n}$ is a sequence of $1$ and $-1$ with the 
condition $\epsilon_1=1$ and $\epsilon_{n}=-1$.

Let $\mu,\mu'\in\{1,-1\}^{\ast}$.
We write a concatenation of two sequences $\mu$ and $\mu'$ as $\mu\circ\mu'$.
Since $\epsilon$ is a sequence of $1$ and $-1$, we abbreviate $\epsilon$ as 
$\epsilon=+^{d_1}-^{d_{2}}+^{d_3}\ldots$ where $+$ (resp. $-$) stands for $1$ (resp. $-1$).

Let $j$ be the integer such that $\epsilon_{i}=1$ for $1\le i\le j-1$ and $\epsilon_{j}=-1$. 
We define $\epsilon':=\epsilon_{j+1}\ldots\epsilon_{n}$.
Let $T$ be the subset of $\{1,\ldots, j-1\}$.
We define a sequence $\epsilon(T)$ by 
\begin{align*}
\epsilon(T):=+^{j-1-d(T)}\circ\epsilon',
\end{align*}
where the number $d(T)$ is defined by
\begin{align*}
d(T)=
\begin{cases}
\#\{k | k\in[1,j-1]\setminus T, k>\min T \}, & \text{ if } T\neq\emptyset, \\
0, & \text{ if } T=\emptyset.
\end{cases}
\end{align*}

\begin{prop}
\label{prop:Whitney}
We have a recurrence relation 
\begin{align}
\label{eq:Whitney}
\mathcal{W}(\epsilon)
=\sum_{T\subseteq[1,j-1]}q^{|T|}\mathcal{W}(\epsilon(T)).
\end{align}
\end{prop}
\begin{proof}
Recall that a directed edge of a network is from a source to a sink.
Since $\epsilon_{j}=-$ is the first sink from left, 
a subset $T\subseteq[1,j-1]$ corresponds to the set of 
directed edges $(i,j)$ where $i\in T$.
Since a network satisfies the property (B1), 
there is no directed edges $(i',j')$ such that 
$i'\in[1,j-1]\setminus T$, $i'>\min T$ and $j'>j$.
Similarly, we may have directed edges $(i,j')$ such that 
$i\in[1,j-1]\setminus T$, $i<\min T$ and $j'>j$.
If we delete the sink $j$ form $\epsilon$, then the maximal number of 
sources left to the sink is given by $j-1-d(T)$.
Therefore, this gives the generating function $\mathcal{W}(\epsilon(T))$.
Note that the exponent of $q$ is the rank of a network, 
which is equivalent to the number of edges, that is, $|T|$.
From these, we have Eq. (\ref{eq:Whitney}).
\end{proof}

\begin{example}
We calculate $\mathcal{W}(++---)$.
By applying Proposition \ref{prop:Whitney}, we have 
\begin{align*}
\mathcal{W}(++---)&=(1+q+q^2)\mathcal{W}(++--)+q \mathcal{W}(+--), \\
&=(1+q+q^2)^2\mathcal{W}(++-)+(1+q+q^2)q\mathcal{W}(+-)+q\mathcal{W}(+--), \\
&=(1+q+q^2)^2(1+q)^2+(1+q+q^2)q(1+q)+q (1+q)^2, \\
&=1+6q+12q^2+13q^3+9q^4+4q^5+q^6.	
\end{align*}
\end{example}

\begin{proof}[Proof of Proposition \ref{prop:evenodd}]
We prove that $\mathcal{P}(n;\epsilon)$ satisfies the statement (E1) 
by induction of the length $l(\epsilon)$ of $\epsilon$.
By a simple calculation, it is obvious that the statement holds true for $l(\epsilon)\le2$.
Assume that the statement (E1) holds true up to $l(\epsilon)=n-1$.
Then, by Proposition \ref{prop:Whitney}, the generating function $\mathcal{W}(\epsilon)$
can be written in terms of $\epsilon(T)$ whose length is strictly smaller than $\epsilon$.
By induction assumption, $\mathcal{W}(\epsilon(T))$ satisfies the statement (E1).
Then, it is obvious that $\mathcal{W}(\epsilon)$ also satisfies the statement (E1), 
which completes the proof.
\end{proof}

\subsection{Forests and Networks}
\label{sec:fornet}
We give a combinatorial interpretation of $\mathcal{W}(\epsilon)$ in 
terms of forests of binary trees.
Let $\epsilon:=\epsilon_1\ldots\epsilon_n=\{1,-1\}^{n}$ satisfying 
$\epsilon_1=1$ and $\epsilon_{n}=-1$.
Let $I^{(\delta)}:=\{i_1,\ldots,i_{l}\}$ with $\delta\in\{1,-1\}$ be the set of indices such that $\epsilon_{i_j}=\delta$  
for $1\le j\le l$ where $l$ is the number of $\delta$ in $\epsilon$. 
We define a weakly decreasing sequence $\lambda(\epsilon)=(\lambda_1,\ldots,\lambda_{r})$ from $\epsilon$ 
as 
\begin{align*}
\lambda_{j}:=\{k|\epsilon_{k}=1, k<i_{r+1-j}\in I^{(-1)}\},
\end{align*}
where $r$ is the number of $-1$ in $\epsilon$ and $1\le j\le r$.
We regard $\lambda$ as a Young diagram in French notation. 
Namely, we place $\lambda_{i}$ cells from bottom to top and left justified.

We introduce a notion of a forest of the Young diagram $\lambda(\epsilon)$.
\begin{defn}
A forest is a Young diagram $\lambda(\epsilon)$ where 
each cell contains either $0$ or $1$ point. 
A cell without (resp. with) a point is called empty (resp. pointed) cell.
A configuration of pointed cells satisfies the following constraint:
\begin{enumerate}[(F1)]
\item For every pointed cell $c$, there may exist a pointed cell below $c$ in the same column, 
or a pointed cell left to $c$ in the same row, but not both.
\end{enumerate}
\end{defn}

\begin{example}
We consider two Young diagrams $\lambda(+-+-)=(2,1)$ and $\lambda(++--)=(2,2)$.
The condition (F1) implies that the left forest in Figure \ref{fig:forest} 
is admissible, but the right one is not allowed.
\begin{figure}[ht]
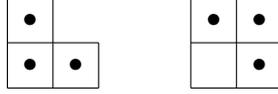

\tikzpic{-0.5}{[scale=0.6]
\draw(0,0)--(0,2)(1,0)--(1,2)(2,0)--(2,1);
\draw(0,0)--(2,0)(0,1)--(2,1)(0,2)--(1,2);
\draw(0.5,0.5)node{$\bullet$}(1.5,0.5)node{$\bullet$}(0.5,1.5)node{$\bullet$};	
}\qquad
\tikzpic{-0.5}{[scale=0.6]
\draw(0,0)--(0,2)(1,0)--(1,2)(2,0)--(2,2);
\draw(0,0)--(2,0)(0,1)--(2,1)(0,2)--(2,2);
\draw(1.5,1.5)node{$\bullet$}(1.5,0.5)node{$\bullet$}(0.5,1.5)node{$\bullet$};	
}
\caption{An example of admissible and non-admissible forests}
\label{fig:forest}
\end{figure}
This is because the pointed cell in the second row and the second column has 
two pointed cells below and left to it.
We have eight forests for $\epsilon=+-+-$ and fourteen forests for $\epsilon=++--$.
\end{example}

Given a forest, we draw two semi-infinite lines from a pointed cell upward 
and rightward.
We say that two lines are crossing if they cross at an empty cell 
in a forest. This empty cell is called a {\it crossing cell}.
Note that if we add a pointed cell on the crossing cell, it violates 
the condition (F1).
There may be several pointed cells on a line starting from a pointed cell.
If we focus on the pointed cells and semi-infinite lines, 
we obtain several binary trees in the Young diagram. 
Since a forest consists of several trees, this is why we call the diagram a forest.

\begin{defn}
We denote by $\mathtt{For}(\epsilon)$ the set of forests associated to 
the sequence $\epsilon\in\{1,-1\}^{*}$.
\end{defn}

\begin{example}
We have two forests which have a crossing cell for $\epsilon=++--$.
They are 
\begin{align}
\label{eq:Forcross}
F_1=
\tikzpic{-0.5}{[scale=0.6]
\draw(0,0)--(0,2)(1,0)--(1,2)(2,0)--(2,2);
\draw(0,0)--(2,0)(0,1)--(2,1)(0,2)--(2,2);
\draw(1.5,0.5)node{$\bullet$}(0.5,1.5)node{$\bullet$};	
\draw[red](1.5,1.5)node{$\square$};
}
\qquad 
F_2=
\tikzpic{-0.5}{[scale=0.6]
\draw(0,0)--(0,2)(1,0)--(1,2)(2,0)--(2,2);
\draw(0,0)--(2,0)(0,1)--(2,1)(0,2)--(2,2);
\draw(0.5,0.5)node{$\bullet$}(1.5,0.5)node{$\bullet$}(0.5,1.5)node{$\bullet$};
\draw[red](1.5,1.5)node{$\square$};
}
\end{align}
where a red square presents a crossing cell.
We have two binary trees in $F_1$, and a unique binary tree in $F_2$. 
\end{example}

Recall $I^{(\pm1)}$ is the set of indices $i$ in $\epsilon\in\{1,-1\}^{\ast}$.
Suppose that a cell $c$ is in the $i$-th row from bottom and in the 
$j$-th column from left.
We define a label of $c$, $l(c)$, as 
$l(c)=(p,q)$ where $p$ is the $j$-the smallest element in $I^{(+1)}$ and 
$q$ is the $i$-th largest element in $I^{(-1)}$.

The next proposition is the characterization of a forest by a network.
\begin{prop}
\label{prop:bijForNet}
A forest $F\in\mathtt{For}(\epsilon)$ is bijective to a network $N\in\mathcal{P}(n;\epsilon)$.
\end{prop}
\begin{proof}
We will construct a bijection between $\mathtt{For}(\epsilon)$ and $\mathcal{P}(n;\epsilon)$.
Given a forest $F$, we define the set of directed edges by
\begin{align*}
\mathcal{E}(F):=\{ (i,j) | (i,j) \text{ is a label of either a pointed or crossing cell} \}.
\end{align*}
By construction of a forest, the directed edges in $\mathcal{E}(F)$ satisfy the conditions from (A1)
to (A4) and (B1). Thus, we have a network $N(F)$.
It is obvious if $F\neq F'$,  then $N(F)\neq N(F')$.

Conversely, suppose we have a network $N\in\mathcal{P}(n;\epsilon)$ and $\mathcal{E}(N)$ is 
the set of directed edges of $N$.
We have a pointed cell corresponding to an element in $\mathcal{E}(N)$.
Pick a pointed cell $c$. Then, if there exist two pointed cells which are left to and below $c$, 
we replace the cell $c$ by a crossing cell.
We continue this process for all cells, then obtain a forest $F(N)$ satisfying the condition
(F1). It is obvious if $N\neq N'$, then $F(N)\neq F(N')$.
From these observations, we have a natural bijection between the two sets.
\end{proof}

\begin{example}
Consider the two diagrams in Eq. (\ref{eq:Forcross}).
The sets of directed edges for the networks for $F_1$ and $F_2$ are given by
\begin{align*}
N(F_1)&=\{(1,3),(2,3),(2,4)\}, \\
N(F_2)&=\{(1,3),(1,4),(2,3),(2,4)\}.
\end{align*}
Note that the crossing cell corresponds to the directed edge $(2,3)$.
\end{example}

A binary tree consists of nodes which have degree two or three.
Here, a degree of a node $\mathtt{n}$ is the number of edges connected to $\mathtt{n}$.
The degree of the root of a tree is two, and that of other internal nodes is three.
Similarly, we define the degree of a crossing cell is four.
We change a connectivity of semi-infinite lines in a forest as in Figure \ref{fig:reconnect}.
\begin{figure}[ht]
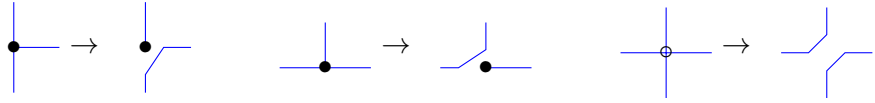

\tikzpic{-0.44}{[scale=0.6]
\draw[blue](0,0)--(0,2)(0,1)node[black]{$\bullet$}--(1,1);
}$\rightarrow$
\tikzpic{-0.44}{[scale=0.6]
\draw[blue](0,1)node[black]{$\bullet$}--(0,2);
\draw[blue](0,0)--(0,0.4)--(0.4,1)--(1,1);
}\qquad
\tikzpic{-0.5}{[scale=0.6]
\draw[blue](0,0)--(2,0)(1,0)node[black]{$\bullet$}--(1,1);
}$\rightarrow$
\tikzpic{-0.5}{[scale=0.6]
\draw[blue](1,0)node[black]{$\bullet$}--(2,0);
\draw[blue](0,0)--(0.4,0)--(1,0.4)--(1,1);
}\qquad
\tikzpic{-0.5}{[scale=0.6]
\draw[blue](0,0)--(2,0)(1,-1)--(1,1);
\draw(1,0)node{$\circ$};
}$\rightarrow$
\tikzpic{-0.5}{[scale=0.6]
\draw[blue](0,0)--(0.6,0)--(1,0.4)--(1,1);
\draw[blue](1,-1)--(1,-0.4)--(1.4,0)--(2,0);
}
\caption{Reconnection of semi-infinite lines at nodes of degree three and four.}
\label{fig:reconnect}
\end{figure}
By this operation, the degree of an internal node which is not the root becomes one, and that of a crossing 
cell becomes zero.
Each semi-infinite line in a forest looks like Figure \ref{fig:int} after the reconnections of lines.
The degree of a pointed cell is either one or two.
\begin{figure}[ht]
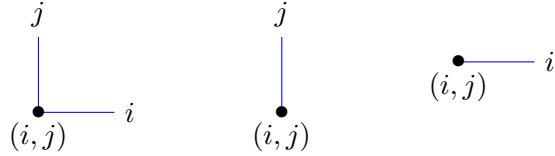

\tikzpic{-0.5}{
\draw[blue](0,1)node[anchor=south,black]{$j$}--(0,0)node[black]{$\bullet$}node[anchor=north,black]{$(i,j)$}
--(1,0)node[anchor=west,black]{$i$};
}
\qquad
\tikzpic{-0.5}{
\draw[blue](0,1)node[anchor=south,black]{$j$}--(0,0)node[black]{$\bullet$}node[anchor=north,black]{$(i,j)$};
}
\qquad
\tikzpic{-0.5}{
\draw[blue](0,0)node[black]{$\bullet$}node[anchor=north,black]{$(i,j)$}
--(1,0)node[anchor=west,black]{$i$};
}
\caption{The label of semi-infinite lines.}
\label{fig:int}
\end{figure}
We assign an integer to an semi-infinite line as follows.
Suppose $(i,j)$ is a label of a node. 
Then, as in Figure \ref{fig:int}, we assign an integer $j$ to 
a vertical line, and $i$ to a horizontal line if exists.
Let $F$ be a forest, and $L(F)$ be the set of labels assigned to 
semi-infinite line.
Define $n:=\lambda_{1}+r$ for the Young diagram $\lambda$, where 
$r$ is the length of $\lambda$.
We read the labels of semi-infinite lines from left-most one in a clockwise way, and 
denote by $w'(F)$ the word obtained in this way.
We will construct a permutation $\pi^{F}:=\pi(1)\ldots\pi(n)\in\mathcal{S}_{n}$ of the set $[n]$ from $w':=w'(F)$
as follows:
\begin{enumerate}[(G1)]
\item If $i\in[n]\setminus L(F)$, then $\pi^{F}(i)=i$.
\item If $L(F):=\{i_1<i_2<\ldots<i_{t}\}$, we define $\pi^{F}(i_j)=w'(j)$ for $1\le j\le t$. 
\end{enumerate}

\begin{defn}
Let $F\in\mathtt{For}(\epsilon)$ . Then, we define a map 
$\kappa:\mathtt{For}(\epsilon)\rightarrow\mathcal{S}_{n}$, 
$F\mapsto\pi^{F}$ given by (G1) and (G2).
\end{defn}

From Proposition \ref{prop:bijForNet}, we have a bijection between 
$\mathtt{For}(\epsilon)$ and $\mathcal{P}(n;\epsilon)$.

\begin{prop}
\label{prop:FortoNet}
Let $F\in\mathcal{F}$ and $N\in\mathcal{P}(n;\epsilon)$ be a forest and a network bijective to 
each other.
We have 
$\kappa(F)=\sigma(N)^{-1}$, i.e., two permutations are inverse of each other.
\end{prop}
\begin{proof}
Since we locally reconnect semi-infinite lines as in Figure \ref{fig:reconnect}, it is enough 
to show that $\kappa(F)=\sigma(N)^{-1}$ around nodes of degree two, three and four.
In the case of degree two, it is obvious that we have $\kappa(F)=\sigma(N)^{-1}$.
We have two cases for degree two.
The label of a node gives a directed edge in the network $\sigma(N)$, 
we have the correspondence between a binary tree with two internal nodes and 
a permutation associated to the binary tree.
Let $i,j,$ and $k$ be integers such that $i<j<k$.
Namely, we have
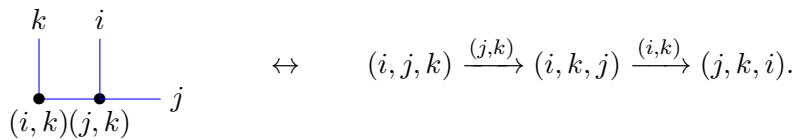
\begin{align*}
\tikzpic{-0.5}{[scale=0.8]
\draw[blue](0,1)to(0,0)node[black]{$\bullet$}node[anchor=north,black]{$(i,k)$}
to(1,0)node[anchor=north,black]{$(j,k)$}node[black]{$\bullet$};
\draw[blue](1,1)--(1,0)--(2,0);
\draw(0,1)node[anchor=south]{$k$}(1,1)node[anchor=south]{$i$}(2,0)node[anchor=west]{$j$};
} 
\qquad\leftrightarrow\qquad
(i,j,k)\xrightarrow{(j,k)}(i,k,j)\xrightarrow{(i,k)}(j,k,i).
\end{align*}
Note that $(k,i,j)$ is the inverse of $(j,k,i)$.
Similarly, we have the correspondence:
\begin{align*}
\tikzpic{-0.5}{[scale=0.8]
\draw[blue](0,2)node[black,anchor=south]{$j$}to(0,1)node[black]{$\bullet$}node[anchor=east,black]{$(i,j)$}
to (0,0)node[black]{$\bullet$}node[anchor=north,black]{$(i,k)$}to(1,0)node[anchor=west,black]{$i$};
\draw[blue](0,1)to(1,1)node[anchor=west,black]{$k$};
}\qquad\leftrightarrow\qquad
(i,j,k)\xrightarrow{(i,j)}(j,i,k)\xrightarrow{(i,k)}(k,i,j).
\end{align*}
Note that $(j,k,i)$ is the inverse of $(k,i,j)$.

We consider the case of degree four. The node of degree four also gives a directed 
edge by definition.
Let $i,j,k$ and $l$ be integers such that $i<j<k<l$.
The correspondence is given by 
\begin{align*}
\tikzpic{-0.5}{
\draw[blue](0,2)node[anchor=south,black]{$k$}to(0,1)node[black]{$\bullet$}node[anchor=east,black]{$(i,k)$}
to(2,1)node[anchor=west,black]{$l$};
\draw[blue](1,2)node[anchor=south,black]{$i$}to(1,0)node[black]{$\bullet$}node[anchor=north,black]{$(j,l)$}
to(2,0)node[anchor=west,black]{$j$};
\draw(1,1)node{$\circ$}node[anchor=south west]{$(j,k)$};
}\qquad\leftrightarrow\qquad
(i,j,k,l)\xrightarrow{(j,k)}(i,k,j,l)\xrightarrow{(j,l)}(i,l,j,k)\xrightarrow{(i,k)}(j,l,i,k).
\end{align*}
Note that $(k,i,l,j)$ is the inverse of $(j,l,i,k)$.

In all cases, we have $\kappa(F)=\sigma(N)^{-1}$.
The locality of the reconnection of semi-infinite lines guarantees that 
$\kappa(F)=\sigma(N)^{-1}$ holds for every forest.
This completes the proof. 
\end{proof}

\begin{example}
Let $\epsilon=(+,+,+,-,-,-)$.
We consider the network for $436215$ as in Figure \ref{fig:tlt}.
The forest for this network is given by the right picture.
In the forest, we have three internal nodes labeled $(1,5)$, $(2,4)$ 
and $(3,6)$, and two crossing cells.
The inverse of $542163$ is $436215$.

\begin{figure}[ht]
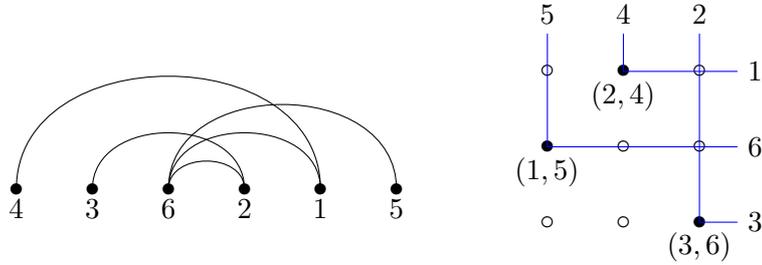

\tikzpic{-0.5}{
\foreach \x/\y in{1/4,2/3,3/6,4/2,5/1,6/5}{
\draw(\x,0)node{$\bullet$}node[anchor=north]{$\y$};
}
\draw(1,0)..controls(1,2)and(5,2)..(5,0);
\draw(2,0)..controls(2,1)and(4,1)..(4,0);
\draw(3,0)..controls(3,0.5)and(4,0.5)..(4,0);
\draw(3,0)..controls(3,1)and(5,1)..(5,0);
\draw(3,0)..controls(3,1.5)and(6,1.5)..(6,0);
}\qquad
\tikzpic{-0.5}{
\foreach \x/\y in {0/0,1/0,1/1,2/1,0/2,2/2}{
\draw(\x,\y)node{$\circ$};
}
\foreach \x/\y in{0/1,1/2,2/0}{
\draw(\x,\y)node{$\bullet$};
}
\draw(0,1)node[anchor=north]{$(1,5)$};
\draw(1,2)node[anchor=north]{$(2,4)$};
\draw(2,0)node[anchor=north]{$(3,6)$};
\draw[blue](0,2.5)node[anchor=south,black]{$5$}--(0,1)--(2.5,1)node[anchor=west,black]{$6$};
\draw[blue](1,2.5)node[anchor=south,black]{$4$}--(1,2)--(2.5,2)node[anchor=west,black]{$1$};
\draw[blue](2,2.5)node[anchor=south,black]{$2$}--(2,0)--(2.5,0)node[anchor=west,black]{$3$};

}

\caption{The network and the forest for $436215$.}
\label{fig:tlt}
\end{figure}
\end{example}

Given a forest $F\in\mathtt{For}(\epsilon)$, we define 
$N(F;\bullet)$ and $N(F;\circ)$  to be the number of pointed cells 
and that of crossing cells respectively.
We consider the ordinary generating function of forests:
\begin{align*}
\mathcal{F}(\epsilon):=
\sum_{F\in\mathtt{For}(\epsilon)}q^{N(F;\bullet)+N(F;\circ)}.
\end{align*}

Recall that $\mathcal{W}(\epsilon)$ is the generating function of 
the Whitney numbers.
\begin{theorem}
We have $\mathcal{W}(\epsilon)=\mathcal{F}(\epsilon)$.
\end{theorem}
\begin{proof}
From Proposition \ref{prop:bijForNet}, we have a natural bijection 
between a forest $F$ and a network $N$.
It is enough to show that $N(F;\bullet)+N(F;\circ)=|\mathcal{E}(N)|$.
However, this equation is obvious from the construction of the 
bijection (see the proof of Proposition \ref{prop:bijForNet}).
This completes the proof.
\end{proof}

\subsection{Forests and permutations}
In Section \ref{sec:fornet}, we see the correspondence between a forest 
and a network. Since a network is bijective to a permutation, this 
correspondence gives a bijection between a forest and a permutation.
In this subsection, we give another correspondence between a forest 
and a permutation, which is compatible with an order of permutations.	

Let $\epsilon\in\{+,-\}^{n}$ be a sequence of $+$ and $-$ such that 
$\epsilon_{1}=+$ and $\epsilon_{n}=-$.
We denote by $N_{\max}(\epsilon)$ the unique network which has the 
maximal number of directed edges.
Recall that $\lambda(\epsilon)$ is a Young diagram obtained from $\epsilon$ in 
French notation.
We define two sets $I^{+}(\epsilon)$ and $I^{-}(\epsilon)$ by 
$I^{\pm}(\epsilon):=\{i\in[n]| \epsilon_{i}=\pm\}$.
We put labels on the west and south edges of $\lambda(\epsilon)$ as follows.
The labels on the west edges are in $I^{-}(\epsilon)$ and increasing from top 
to bottom. Similarly, the labels on the south edges are in $I^{+}(\epsilon)$ and 
increasing from left to right.
A label of a cell $c$ in $\lambda(\epsilon)$ is a pair of integers $(x,y)$ 
where $x$ (resp. $y$) is the label of the south (resp. west) edge below 
(resp. left to) $c$ in the same column (resp. row) in $\lambda(\epsilon)$.
We consider a forest in $\lambda(\epsilon)$ as in Section \ref{sec:fornet}.

Since a forest $F$ in $\lambda(\epsilon)$ satisfies the condition (F1), we have several binary trees 
in the forest by connecting pointed cells by vertical and horizontal lines.
A binary tree consists of nodes and leaves. 
An internal node is a node which has a child node, and a leaf is a node which does not have 
a child node. 
We construct a permutation $\nu(F)$ form $F$ in the following way.
Pick a leaf of a binary tree in $F$, that is, a pointed cell $c$ which has no pointed cell above
and right to it. 
We exchange the labels on the boundary of $\lambda(\epsilon)$ which correspond to the label of $c$,
and delete the pointed cell $c$ from $F$.
We denote by $F'$ the new forest $F$. We write this relation by $F\xrightarrow{c}F'$.
We have a sequence of forests
\begin{align*}
F\xrightarrow{c_1}F_1\xrightarrow{c_2}\cdots\xrightarrow{c_m}F_{m},
\end{align*}
where $F_{m}$ is the forest without pointed cells.	 
By reading the labels of the west and south edges of $F_{m}$ counterclockwise, we obtain a 
permutation $\nu(F)$.

\begin{example}
\label{ex:forest}
We consider a forest $F$ with five pointed cells in $\lambda=(3,3,2)$:
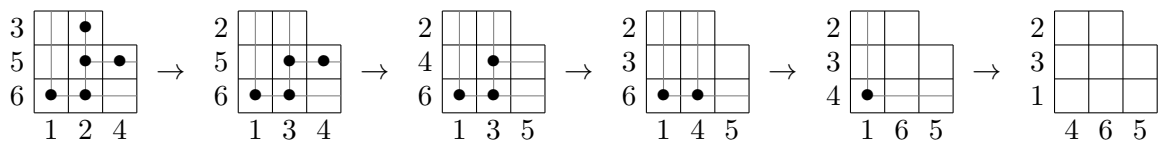
\begin{align*}
\tikzpic{-0.5}{[scale=0.45]
\draw(0,0)--(3,0)(0,1)--(3,1)(0,2)--(3,2)(0,3)--(2,3);
\draw(0,0)--(0,3)(1,0)--(1,3)(2,0)--(2,3)(3,0)--(3,2);
\draw[gray](0.5,0.5)--(0.5,3)(0.5,0.5)--(3,0.5)(1.5,0.5)--(1.5,3)(1.5,1.5)--(3,1.5);
\draw(0.5,0.5)node{$\bullet$}(1.5,0.5)node{$\bullet$}(1.5,1.5)node{$\bullet$}
(2.5,1.5)node{$\bullet$}(1.5,2.5)node{$\bullet$};
\draw(-0.5,2.5)node{$3$}(-0.5,1.5)node{$5$}(-0.5,0.5)node{$6$};
\draw(0.5,-0.5)node{$1$}(1.5,-0.5)node{$2$}(2.5,-0.5)node{$4$};
}
\rightarrow
\tikzpic{-0.5}{[scale=0.45]
\draw(0,0)--(3,0)(0,1)--(3,1)(0,2)--(3,2)(0,3)--(2,3);
\draw(0,0)--(0,3)(1,0)--(1,3)(2,0)--(2,3)(3,0)--(3,2);
\draw[gray](0.5,0.5)--(0.5,3)(0.5,0.5)--(3,0.5)(1.5,0.5)--(1.5,3)(1.5,1.5)--(3,1.5);
\draw(0.5,0.5)node{$\bullet$}(1.5,0.5)node{$\bullet$}(1.5,1.5)node{$\bullet$}
(2.5,1.5)node{$\bullet$};
\draw(-0.5,2.5)node{$2$}(-0.5,1.5)node{$5$}(-0.5,0.5)node{$6$};
\draw(0.5,-0.5)node{$1$}(1.5,-0.5)node{$3$}(2.5,-0.5)node{$4$};
}
\rightarrow
\tikzpic{-0.5}{[scale=0.45]
\draw(0,0)--(3,0)(0,1)--(3,1)(0,2)--(3,2)(0,3)--(2,3);
\draw(0,0)--(0,3)(1,0)--(1,3)(2,0)--(2,3)(3,0)--(3,2);
\draw[gray](0.5,0.5)--(0.5,3)(0.5,0.5)--(3,0.5)(1.5,0.5)--(1.5,3)(1.5,1.5)--(3,1.5);
\draw(0.5,0.5)node{$\bullet$}(1.5,0.5)node{$\bullet$}(1.5,1.5)node{$\bullet$};
\draw(-0.5,2.5)node{$2$}(-0.5,1.5)node{$4$}(-0.5,0.5)node{$6$};
\draw(0.5,-0.5)node{$1$}(1.5,-0.5)node{$3$}(2.5,-0.5)node{$5$};
}
\rightarrow
\tikzpic{-0.5}{[scale=0.45]
\draw(0,0)--(3,0)(0,1)--(3,1)(0,2)--(3,2)(0,3)--(2,3);
\draw(0,0)--(0,3)(1,0)--(1,3)(2,0)--(2,3)(3,0)--(3,2);
\draw[gray](0.5,0.5)--(0.5,3)(0.5,0.5)--(3,0.5)(1.5,0.5)--(1.5,3);
\draw(0.5,0.5)node{$\bullet$}(1.5,0.5)node{$\bullet$};
\draw(-0.5,2.5)node{$2$}(-0.5,1.5)node{$3$}(-0.5,0.5)node{$6$};
\draw(0.5,-0.5)node{$1$}(1.5,-0.5)node{$4$}(2.5,-0.5)node{$5$};
}
\rightarrow
\tikzpic{-0.5}{[scale=0.45]
\draw(0,0)--(3,0)(0,1)--(3,1)(0,2)--(3,2)(0,3)--(2,3);
\draw(0,0)--(0,3)(1,0)--(1,3)(2,0)--(2,3)(3,0)--(3,2);
\draw[gray](0.5,0.5)--(0.5,3)(0.5,0.5)--(3,0.5);
\draw(0.5,0.5)node{$\bullet$};
\draw(-0.5,2.5)node{$2$}(-0.5,1.5)node{$3$}(-0.5,0.5)node{$4$};
\draw(0.5,-0.5)node{$1$}(1.5,-0.5)node{$6$}(2.5,-0.5)node{$5$};
}
\rightarrow
\tikzpic{-0.5}{[scale=0.45]
\draw(0,0)--(3,0)(0,1)--(3,1)(0,2)--(3,2)(0,3)--(2,3);
\draw(0,0)--(0,3)(1,0)--(1,3)(2,0)--(2,3)(3,0)--(3,2);
\draw(-0.5,2.5)node{$2$}(-0.5,1.5)node{$3$}(-0.5,0.5)node{$1$};
\draw(0.5,-0.5)node{$4$}(1.5,-0.5)node{$6$}(2.5,-0.5)node{$5$};
}
\end{align*}
By reading the labels on the boundary of the Young diagram, we have 
the permutation $\nu(F)=231465$.
Note that we have two leaves whose labels are $(2,3)$ and $(4,5)$, and 
the order to delete these cells is irrelevant to the permutation $\nu(F)$.
\end{example}

We give another characterization of the permutation $\nu(F)$ from $F$.
Let $\pi$ be a permutation corresponding to the network with maximal number of directed 
edges.
The permutation $\nu(F)$ is also obtained from the forest $F$ in a similar way to $\kappa(F)$.
As in Section \ref{sec:fornet}, we reconnect the semi-infinite line 
from a pointed cell as the left and middle pictures in Figure \ref{fig:reconnect}.
We do not reconnect the lines of degree four.
We obtain a permutation $\widetilde{\kappa}(F)$ from $F$ in a similar manner by 
use of (G1) and (G2).
Define a permutation $\mu$ by 
\begin{align}
\label{eq:defmu}
\mu(F):=\widetilde{\kappa}(F)\circ \pi^{-1}.
\end{align}
where $u\circ v$ is a permutation product of $u$ and $v$.

\begin{prop}
\label{prop:numu}
We have 
\begin{align}
\label{eq:numu}
\nu(F)=\mu(F)^{-1}.
\end{align}
\end{prop}
\begin{proof}
We prove the statement by induction on the number of pointed cells.
Suppose that $F$ is a forest without pointed cells. 
It is clear that $\widetilde{\kappa}(F)=\mathrm{id}$ and the reading word 
of the labels on the boundary is $\pi$.
We have $\nu(F)=(\widetilde{\kappa}(F)\circ \pi^{-1})^{-1}=\pi$, which 
implies Eq. (\ref{eq:numu}).

Suppose that Eq. (\ref{eq:numu}) holds for a forest $F'$ with $m-1$ pointed cells.
A forest $F'$ with $m-1$ pointed cells can be obtained from a forest $F$ with $m$ pointed cells
by deleting a root of a binary tree of $F$. 
Since we may have several binary trees in $F$, there are several choices of $F'$.
By induction hypothesis, we have $\nu(F')=\mu(F')^{-1}$.
We add one pointed cell $c$ to $F'$.
To compute $\nu(F)$, we need to consider semi-infinite lines starting from pointed cells, 
and reconnect them according to the left and middle pictures in Figure \ref{fig:reconnect}.
Recall that we have a sequence of forests
\begin{align*}
F\xrightarrow{c_1} F_{1}\xrightarrow{c_2}\cdots F_{m-1}\xrightarrow{c}F_{m}. 
\end{align*}
Since the last cell in the above sequence is $c$, the labels on the boundary of the Young diagram 
for $F_{m-1}$ is nothing but the permutation $\nu(F')$.
We need to exchange the labels corresponding to $c$.
By a diagram chasing as in the proof of Proposition \ref{prop:FortoNet},
it is clear that we have $\nu(F)=\mu(F)^{-1}$ by use of $\nu(F')=\mu(F')^{-1}$.
\end{proof}

\begin{example}
Consider the same forest $F$ as in Example \ref{ex:forest}.
It is easy to see that $\widetilde{\kappa}(F)=635142$. 
Since we have $\pi=356124$, $\pi^{-1}=451623$.
From these, $\mu(F)=635142\circ 451623=312465$.
From Proposition \ref{prop:numu}, we obtain the permutation
$\nu(F)=\mu(F)^{-1}=(312465)^{-1}=231465$.
This is nothing but the same permutation in Example \ref{ex:forest}.
\end{example}

In Section \ref{sec:fornet}, we have seen that the generating function $\mathcal{F}(\epsilon)$ 
involves both the numbers of pointed cells and crossing cells.
We will see that the permutation $\nu(F)$ reflects only the number of pointed cells.
To show this, we interpret the number of pointed cells in terms of the length 
of a chain of permutations.

Let $\pi,\nu\in\mathcal{S}_{n}$ be permutations. 
We say that $\nu$ covers $\pi$ if there exists a pair of integers $(i,j)$ such that
$\pi_{i}>\pi_{j}$, $(\nu_i,\nu_j)=(\pi_{j},\pi_{i})$, and $\nu_{k}=\pi_{k}$ for 
$i,j\neq k$.
We write $\pi\lessdot\nu$ when $\nu$ covers $\pi$.

We define a graded set $B(\pi):=\bigcup_{i\ge0}B_{i}(\pi)$ as follows.
First, define $B_{0}(\pi)=\{\pi\}$.
Secondly, we define the sets $B_{i\ge1}(\pi)$ recursively by 
\begin{align*}
B_{i+1}(\pi):=\{\nu | B_{i}(\pi)\ni\pi'\lessdot\nu  \}\setminus \bigcup_{0\le j\le i}B_{j}(\pi),
\end{align*}
where $i\ge0$.
Given a permutation $\pi$, we can consider a graded poset with the order described as above.
This poset is not in general a lattice.
The posets for $\pi=321$ and $\pi=312$ are depicted in Figure \ref{fig:poset321}.
\begin{figure}[ht]
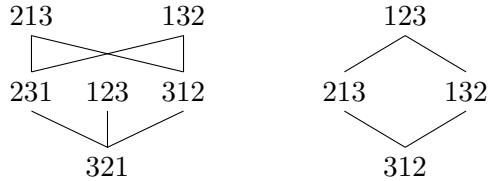

\tikzpic{-0.5}{
\node(c0)at(0,0){$321$};
\node(c1)at(0,1){$123$};
\node(c2)at(-1,1){$231$};
\node(c3)at(1,1){$312$};
\node(c4)at(-1,2){$213$};
\node(c5)at(1,2){$132$};
\draw(c0.north)--(c1.south)(c0.north)--(c2.south)(c0.north)--(c3.south)
(c2.north)--(c4.south)(c2.north)--(c5.south)(c3.north)--(c4.south)(c3.north)--(c5.south);
}
\qquad
\tikzpic{-0.5}{
\node(c3)at(0,0){$312$};
\node(c4)at(-0.8,1){$213$};
\node(c5)at(0.8,1){$132$};
\node(c1)at(0,2){$123$};
\draw(c3.north)--(c4.south)(c3.north)--(c5.south)(c4.north)--(c1.south)(c5.north)--(c1.south);
}
\caption{Posets whose minimal elements are $321$ and $312$.}
\label{fig:poset321}
\end{figure}
Suppose that $\nu\in B_{r}(\pi)$. Then, we define the length of $\nu$ from $\pi$ 
by $l(\pi,\nu):=r$.
For example, we have $l(321,123)=1$ for $\pi=321$. Similarly, we have $l(312,123)=2$ for $\pi=312$.
Note that $123$ covers two permutations $213$ and $132$, but has already covered $321$ in the first example, 
however, $123$ covers these two permutations in the second example.

Denote by $|F|$ the number of pointed cells in $F$, and let $\nu(F)$ be the permutation 
obtained from $F$.
\begin{prop}
Let $\pi$ be a permutation corresponding to the network $N_{\max}(\epsilon)$.
We have 
\begin{align}
\label{eq:Finl}
|F|=l(\pi,\nu(F)).
\end{align}
\end{prop}
\begin{proof}
Since $\nu(F)=\mu(F)^{-1}$ from Proposition \ref{prop:numu}, 
we have 
\begin{align}
\label{eq:lnuF}
l(\pi,\nu(F))=l(\pi,\pi\circ\widetilde{\kappa}(F)^{-1}),
\end{align}
where we have used the definition of $\mu$ given in Eq. (\ref{eq:defmu}). 
The right hand side of Eq. (\ref{eq:lnuF}) is equal to the 
number of transpositions in $\widetilde{\kappa}(F)^{-1}$.
We consider the reconnection of semi-infinite lines which start from 
pointed cells as in Figure \ref{fig:reconnect}.
This reconnection corresponds to a transposition in $\widetilde{\kappa}(F)^{-1}$.
We have $|F|$ pointed cells in $F$, which implies that 
the number of transpositions in $\widetilde{\kappa}(F)^{-1}$ is equal 
to $|F|$.
As a summary, we have Eq. (\ref{eq:Finl}).
\end{proof}

\begin{example}
Consider the forest 
\begin{align*}
F=
\tikzpic{-0.5}{[scale=0.45]
\draw(0,0)--(3,0)(0,1)--(3,1)(0,2)--(3,2)(0,3)--(2,3);
\draw(0,0)--(0,3)(1,0)--(1,3)(2,0)--(2,3)(3,0)--(3,2);
\draw(-0.5,2.5)node{$3$}(-0.5,1.5)node{$5$}(-0.5,0.5)node{$6$};
\draw(0.5,-0.5)node{$1$}(1.5,-0.5)node{$2$}(2.5,-0.5)node{$4$};
\draw(0.5,0.5)node{$\bullet$}(1.5,1.5)node{$\bullet$}(2.5,0.5)node{$\bullet$}
(0.5,2.5)node{$\bullet$};
}
\end{align*}
We have two binary trees in $F$, and $\nu(F)=123456$ by a simple calculation. 
In fact, we have a sequence of permutations
\begin{align*}
356124\xrightarrow{(1,3)}156324\xrightarrow{(2,5)}126354
\xrightarrow{(4,6)}124356\xrightarrow{(3,4)}123456.
\end{align*}
The number of transpositions is four, which is equal to the number of 
pointed cells in $F$.
Note that the we have several sequences of permutations from $356124$ to 
$123456$, but we always have $l(356124,123456)=4$.
\end{example}

\begin{remark}
Given a forest $F$, we have two permutations for $F$: one is $\kappa(F)$, and 
the other is $\nu(F)$. 
The map $\kappa$ reflects the sum of the numbers of pointed cells and crossing cells in $F$, 
that is, the number of directed edges in the corresponding network.
On the other hand, $\nu(F)$ reflects only the number of pointed cells in $F$.
This difference comes from taking into account a reconnection of semi-infinite lines 
of degree four as in Figure \ref{fig:reconnect}, or not.
\end{remark}

\section{Shellability and M\"obius functions}
\label{sec:shell}
We briefly recall the notions related to the shellability 
following \cite{Bjo80,BjoGarSta82,BjoWac83}. 
Let $P$ be a poset and denote by $C(P)$ the covering relations, 
$C(P):=\{(x,y)\in P\times P | x\lessdot y\}$.
An {\it edge-labeling} of $P$ is a map $\lambda:C(P)\rightarrow\Lambda$ where
$\Lambda$ is some poset. In this paper, we consider only the case 
$\Lambda=\mathbb{N}$.
We assign a non-negative integer to an each edge of the Hasse diagram of $P$.
Let $c:x_0\lessdot x_1\lessdot\ldots\lessdot x_{k}$ be an unrefinable chain 
in $P$.
An edge-labeling $\lambda$ is called {\it rising} if
$\lambda(x_{0},x_1)\le \lambda(x_1,x_2)\le \ldots\le \lambda(x_{k-1},x_k)$.

\begin{defn}[Definition 2.1 in \cite{Bjo80}]
\label{defn:RLE}
We define an $R$-labeling and $EL$-labeling as follows.
\begin{enumerate}
\item An edge-labeling $\lambda$ is an $R$-labeling if 
there exists a unique unrefinable chain $c:x=x_0\lessdot x_1\lessdot\ldots\lessdot x_k=y$
whose edge-labeling is rising for any interval $[x,y]$ in P.
\item $\lambda$ is called an $EL$-labeling if
\begin{enumerate}
\item $\lambda$ is an $R$-labeling,
\item for every interval $[x,y]$, there is a unique unrefinable chain $c$ and if 
$x\lessdot z\le y$ and $z\neq x_1$, then $\lambda(x,x_1)<\lambda(x,z)$.
\end{enumerate}
\end{enumerate}
\end{defn}
The condition (2b) means that the unique rising chain $c$ is lexicographically 
first compared to other chains.

\begin{defn}[\cite{Bjo80,BjoWac83}]
\label{defn:LEshellable}
A poset is lexicographically shellable if it is graded and 
admits an $EL$-labeling.
\end{defn}

To show that $\mathcal{P}(n;\epsilon)$ is shellable, 
we will construct an explicit $EL$-labeling on the lattice $\mathcal{P}(n;\epsilon)$.
Suppose $x\lessdot y$.
The edge-labeling $\lambda(x,y)$ is given by 
\begin{align}
\label{eq:lambda}
\lambda(x,y):=\mathcal{E}(y)\setminus\mathcal{E}(x),
\end{align}
where $\mathcal{E}(x)$ is the set of directed edges in $x$.
This definition is well-defined since $|\mathcal{E}(y)|=|\mathcal{E}(x)|+1$ and $x\lessdot y$.

Let $\overline{\mathcal{E}(\epsilon)}$ be the set of directed edges 
in $N_{\mathrm{max}}(\epsilon)$.
We define a linear order on the directed edges in $\overline{\mathcal{E}(\epsilon)}$ as follows.

\begin{defn}
Suppose $(i,j),(k,l)\in\overline{\mathcal{E}(\epsilon)}$. 
Then, we define an order of directed edges by 
\begin{align}
\label{eq:defLambda}
(i,j)<(k,l),
\end{align}
if $j<l$, or if $j=l$ and $i>k$.
\end{defn}

\begin{example}
Let $\epsilon=+-++--$. 
We have seven possible directed edges associated to $\epsilon$, i.e.,
$|\overline{\mathcal{E}(\epsilon)}|=7$.
We have the following order of labels: 
\begin{align*}
(1,2)<(4,5)<(3,5)<(1,5)<(4,6)<(3,6)<(1,6).
\end{align*}	
\end{example}

We consider a subposet which has a crossing as in Figure \ref{fig:LE}.
The integer labels $1,2$ and $3$ stand for the directed edges
$(2,3)$, $(1,3)$ and $(2,4)$ respectively.
\begin{figure}[ht]
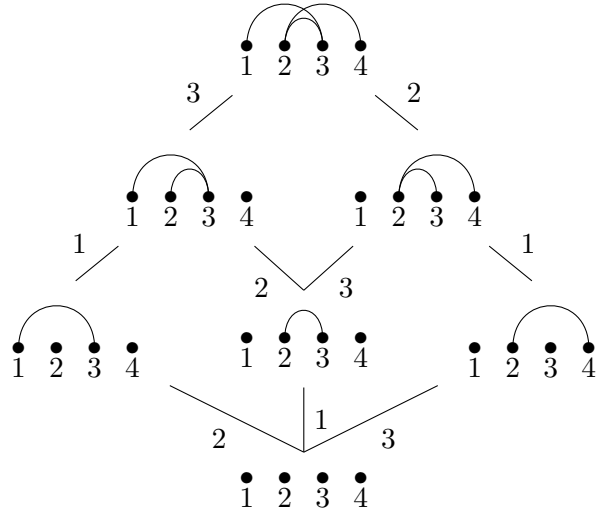

\tikzpic{-0.5}{
\node(0)at(0,-6){
\tikzpic{-0.5}{[scale=0.5]
\foreach \x in {1,2,3,4}{
\draw(\x,0)node{$\bullet$}node[anchor=north]{$\x$};
}}};
\node(1)at(-3,-4){
\tikzpic{-0.5}{[scale=0.5]
\foreach \x in {1,2,3,4}{
\draw(\x,0)node{$\bullet$}node[anchor=north]{$\x$};
}
\draw(1,0)..controls(1,1.5)and(3,1.5)..(3,0);
}};
\node(2)at(0,-4){
\tikzpic{-0.5}{[scale=0.5]
\foreach \x in {1,2,3,4}{
\draw(\x,0)node{$\bullet$}node[anchor=north]{$\x$};
}
\draw(2,0)..controls(2,1)and(3,1)..(3,0);
}};
\node(3)at(3,-4){
\tikzpic{-0.5}{[scale=0.5]
\foreach \x in {1,2,3,4}{
\draw(\x,0)node{$\bullet$}node[anchor=north]{$\x$};
}
\draw(2,0)..controls(2,1.5)and(4,1.5)..(4,0);
}};
\node(4)at(-1.5,-2){
\tikzpic{-0.5}{[scale=0.5]
\foreach \x in {1,2,3,4}{
\draw(\x,0)node{$\bullet$}node[anchor=north]{$\x$};
}
\draw(1,0)..controls(1,1.5)and(3,1.5)..(3,0);
\draw(2,0)..controls(2,1)and(3,1)..(3,0);
}};
\node(5)at(1.5,-2){
\tikzpic{-0.5}{[scale=0.5]
\foreach \x in {1,2,3,4}{
\draw(\x,0)node{$\bullet$}node[anchor=north]{$\x$};
}
\draw(2,0)..controls(2,1.5)and(4,1.5)..(4,0);
\draw(2,0)..controls(2,1)and(3,1)..(3,0);
}};
\node(6)at(0,0){
\tikzpic{-0.5}{[scale=0.5]
\foreach \x in {1,2,3,4}{
\draw(\x,0)node{$\bullet$}node[anchor=north]{$\x$};
}
\draw(1,0)..controls(1,1.5)and(3,1.5)..(3,0);
\draw(2,0)..controls(2,1.5)and(4,1.5)..(4,0);
\draw(2,0)..controls(2,1)and(3,1)..(3,0);
}};
\draw(0.north)to node[anchor=north east]{$2$}(1);
\draw(0.north)to node[anchor=west]{$1$}(2);
\draw(0.north)to node[anchor=north west]{$3$}(3);
\draw(1.north)to node[anchor=south east]{$1$}(4);
\draw(2.north)to node[anchor=north east]{$2$}(4);
\draw(2.north)to node[anchor=north west]{$3$}(5);
\draw(3.north)to node[anchor=south west]{$1$}(5);
\draw(4.north)to node[anchor=south east]{$3$}(6);
\draw(5.north)to node[anchor=south west]{$2$}(6);
}
\caption{A subposet which has a crossing.}
\label{fig:LE}
\end{figure}
It is clear that the labels in the subposet in Figure \ref{fig:LE} 
give an $EL$-labeling. Note that we have no decreasing chain from $\hat{0}$ to $\hat{1}$.
A subposet in $P(n;\epsilon)$ is not in general Eulerian as in Figure \ref{fig:LE}.
In some cases, a subposet is a Boolean lattice, and hence Eulerian.
We come back to this point when we compute the M\"obius function of an interval 
$[x,y]$ in $P(n;\epsilon)$.

\begin{lemma}
\label{lemma:EL}
A edge-labeling $\lambda$ defined in Eq. (\ref{eq:lambda}) is an $EL$-labeling.
\end{lemma}
\begin{proof}
We first show that $\lambda$ is an $R$-labeling.
Since a network has no loops and multiple edges, each directed edge appears 
exactly once in a chain of $[\hat{0},\hat{1}]$.
For a crossing edge $(j,k)$, we always have the order $(j,k)<(i,k)<(j,l)$ with 
$i<j<k<l$. 
Since this order is compatible with the order of edge-labels, we have a unique 
rising chain for any interval $[x,y]$. Thus, $\lambda$ is an $R$-labeling.

Secondly, we show that $\lambda$ satisfies the condition (2b) in Definition \ref{defn:RLE}.
By the same reason as above, the unique rising chain is lexicographically first 
compared to other chains. This completes the proof.
\end{proof}

Definition \ref{defn:LEshellable} and Lemma \ref{lemma:EL} imply the following.
\begin{theorem}
\label{thrm:leshellable}
The lattice $\mathcal{P}(n;\epsilon)$ is lexicographically shellable.
\end{theorem}

A direct consequence of Theorem \ref{thrm:leshellable} (see also \cite{Bjo80}) is the 
following corollary.
\begin{cor}
\label{cor:CM}
The lattice $\mathcal{P}(n;\epsilon)$ is shellable, hence Cohen--Macaulay.
\end{cor}

The first application of the $EL$-labeling introduced above is to show 
that the interval $[x,y]$ in $\mathcal{P}(n;\epsilon)$ is supersolvable.
We recall the definition of $\mathcal{S}_{n}$ $EL$-labeling, or snelling for short.
\begin{defn}[Definition 2.2 in \cite{McNam03}]
An $EL$-labeling $\lambda$ of $P$ is said to be an {\it $\mathcal{S}_{n}$ $EL$-labeling}, 
or {\it snelling}, if the map from $i$ to $\lambda(x_{i-1},x_{i})$ is a permutation of $[n]$ for 
every maximal chain $\hat{0}=x_0<x_1<\cdots<x_{n}=\hat{1}$.
\end{defn}

\begin{lemma}
\label{lemma:snelling}
Let $[x,y]$ be an interval in $\mathcal{P}(n;\epsilon)$. 
The $EL$-labeling for $[x,y]$ is snelling.
\end{lemma}
\begin{proof}
Let $\mathcal{E}(y\backslash x)=\mathcal{E}(y)\setminus\mathcal{E}(x)$ be the 
set of directed edges for $[x,y]$.
From the definition of the covering relation, any edge in $\mathcal{E}(y\backslash x)$
appears exactly once in a maximal chain in $[x,y]$ as an edge-label.
We have a linear order of directed edges as in Eq. (\ref{eq:defLambda}), each 
maximal chain gives a permutation in $[n]$ where $n:=\rho(x,y)$.
\end{proof}

\begin{defn}[Definition 1.1 in \cite{Sta72}, \cite{McNam03}]
A finite lattice  $L$ is said to be {\it supersolvable} if it contains a 
maximal chain, called an $M$-chain of $L$, which together with 
any other chain in $L$ generates a distributive sublattice.
\end{defn}
One of the main results in \cite{McNam03} is as follows.
\begin{theorem}[Theorem 1 in \cite{McNam03}]
\label{thrm:SS}
A finite graded lattice of rank $n$ is supersolvable if and only if 
it is $\mathcal{S}_{n}$ $EL$-shellable.
\end{theorem}

\begin{cor}
An interval in $\mathcal{P}(n;\epsilon)$ is supersolvable.
\end{cor}
\begin{proof}
From Theorem \ref{thrm:leshellable}, any interval $[x,y]$ in $\mathcal{P}(n;\epsilon)$ 
is $EL$-shellable. By Lemma \ref{lemma:snelling}, $[x,y]$ is snelling.
These imply that $[x,y]$ is $\mathcal{S}_{n}$ $EL$-shellable.
From Theorem \ref{thrm:SS}, $[x,y]$ is supersolvable.
\end{proof}

As the second application of the $EL$-labeling on the poset, we compute the M\"obius 
function of any interval $[x,y]$ in $P(n;\epsilon)$.
Let $L$ be a lattice. 
Then, the M\"obius function of a lattice $L$, $\mu:L\times L\rightarrow\mathbb{Z}$, 
is defined recursively by
\begin{align*}
\mu(x,y):=
\begin{cases}
1, & \text{ if } x=y, \\
-\sum_{x\le z<y}\mu(x,z), & \text{ if } x<y.
\end{cases}
\end{align*}
We define $\mu(P):=\mu(\hat{0},\hat{1})$.

The M\"obius function of $P$ and an edge-labeling are related as follows.
\begin{prop}[\cite{Bjo80,BjoGarSta82,Sta74}]
\label{prop:MobiusRlabel}
Suppose a poset $P$ admits an $R$-labeling $\lambda$.
When $x\le y$ in $P$, the value $(-1)^{\rho(x,y)}\mu(x,y)$ is equal to
the number of chains $x=x_0\lessdot x_1\lessdot\ldots\lessdot x_k=y$ 
such that
\begin{align}
\label{eq:dec}
\lambda(x_{0},x_1)\not\le\lambda(x_1,x_2)\not\le\ldots\not\le\lambda(x_{k-1},x_k).
\end{align}
\end{prop}
Since we consider only $\Lambda=\mathbb{N}$, the condition (\ref{eq:dec}) 
is equivalent to
\begin{align*}
\lambda(x_{0},x_1)>\lambda(x_1,x_2)>\ldots>\lambda(x_{k-1},x_k).
\end{align*}

Let $x\le y$ be two elements in $\mathcal{P}(n;\epsilon)$.
We define the set of directed edges $\mathcal{E}^{\times}(y\backslash x)$ by
\begin{align*}
\mathcal{E}^{\times}(y\backslash x):=\{(j,k)\notin\mathcal{E}(x) | (i,k),(j,l)\in\mathcal{E}(y), i<j<k<l \}.
\end{align*}
In other words, $\mathcal{E}^{\times}(y\backslash x)$ is the set of edges which exist due to crossings of directed 
edges and in $\mathcal{E}(y)$ but not in $\mathcal{E}(x)$.
Consider the following statement:
\begin{enumerate}[(H1)]
\item $\mathcal{E}^{\times}(y\backslash x)\neq \emptyset$.
\end{enumerate}
Then, we can calculate the M\"obius functions for any interval $[x,y]$ in $\mathcal{P}(n;\epsilon)$.
\begin{theorem}
Let $x\le y$ be two elements in $\mathcal{P}(n;\epsilon)$.
The M\"obius function $\mu(x,y)$ is given by 
\begin{align}
\label{eq:Moebius}
\mu(x,y):=\begin{cases}
0, & \text{if  (H1) holds true}, \\
(-1)^{\rho(x,y)}, & \text{otherwise}.
\end{cases}
\end{align}
\end{theorem}
\begin{proof}
Suppose $\mathcal{E}^{\times}(y\backslash x)=\emptyset$.
A maximal chain from $x$ to $y$ has its edge-labels in 
$\mathcal{E}(y\backslash x):=\mathcal{E}(y)\setminus\mathcal{E}(x)$, and 
there is no constraint on the order of the directed edges.
From Proposition \ref{prop:Boolean}, the interval $[x,y]$ is isomorphic to the Boolean lattice, 
and hence Eulerian. Then, we have $\mu(x,y)=(-1)^{\rho(x,y)}$.

Suppose $x$ and $y$ satisfy the condition (H1).
Let $c$ be a maximal chain from $x$ to $y$. 
From the condition (H1), there exists an edge $(j,k)$ such that 
$(j,k)\in\mathcal{E}(y\backslash x), \mathcal{E}(y)$, and $(i,k)$ and $(j,l)$ are also in 
$\mathcal{E}(y)$ for $i<j<k<l$.
The order of these three edges are 
\begin{align}
\label{eq:ordcross}
(j,k)<(i,k)<(j,l),
\end{align}
by Eq. (\ref{eq:defLambda}).
By the definition of the covering relation on $P(n;\epsilon)$, 
the edge label $(j,k)$ is followed by an edge label $(i,k)$ or $(j,l)$, or 
by both in the chain $c$.
This observation and Eq. (\ref{eq:ordcross}) imply that the chain $c$ cannot be 
a decreasing chain.
The interval from $x$ to $y$ does not have a decreasing chain.
From Proposition \ref{prop:MobiusRlabel}, we have $\mu(x,y)=0$, 
which completes the proof.
\end{proof}

\begin{example}
Consider the poset in Figure \ref{fig:LE}.
By a simple calculation, we have $\mu(\hat{0},\hat{1})=0$ and $\mu(x,y)=(-1)^{\rho(x,y)}$ for all  
$(x,y)=(\hat{0},y)$ with $y\neq\hat{1}$.
The directed edge $(2,3)$ in the poset is a crossing edge.
\end{example}

\bibliographystyle{amsplainhyper} 
\bibliography{biblio}

\providecommand{\bysame}{\leavevmode\hbox to3em{\hrulefill}\thinspace}
\begin{thebibliography}{10}

\bibitem{AvaBouNad13}
J.-C. Aval, A.~Boussicault, and P.~Nadeau, \emph{Tree-like tableaux}, Electron.
  J. Comb. \textbf{20} (2013), no.~4, P34,
  \href{http://dx.doi.org/https://doi.org/10.37236/3440}{\path{doi}}.

\bibitem{Bjo80}
A.~Bj\"orner, \emph{Shellable and {C}ohen-{M}acaulay partially ordered sets},
  Trans. Amer. Math. Soc. \textbf{260} (1980), no.~1, 159--183,
  \href{http://dx.doi.org/https://doi.org/10.1090/S0002-9947-1980-0570784-2}{\path{doi}}.

\bibitem{BjoGarSta82}
A.~Bj\"orner, A.~Garsia, and R.~Stanley, \emph{An introduction to
  {C}ohen-{M}acaulay partially ordered sets}, Ordered Sets (I.~Rival, ed.),
  NATO Adv. Study Inst. Series, Ser. C: Math. Phys. Sci., Reidel, Dordrecht,
  1982.

\bibitem{BjoWac83}
A.~Bj\"orner and M.~L. Wachs, \emph{On lexicographically shellable posets},
  Trans. Amer. Math. Soc. \textbf{277} (1983), no.~1, 323--341,
  \href{http://dx.doi.org/https://doi.org/10.1090/S0002-9947-1983-0690055-6}{\path{doi}}.

\bibitem{KeyWil12}
R.~W. Kenyon and D.~B. Wilson, \emph{{D}ouble-{D}imer {P}airings and {S}kew
  {Y}oung {D}iagrams}, Electron. J. Comb. \textbf{18} (2011), P130,
  \href{http://dx.doi.org/https://doi.org/10.37236/617}{\path{doi}}.

\bibitem{KimMesPanWil14}
J.~S. Kim, K.~M\'eszáros, G.~Panova, and D.~B. Wilson, \emph{Dyck tilings,
  increasing trees, descents, and inversions}, J. Comb. Theory Ser. A.
  \textbf{122} (2014), 9--27,
  \href{http://dx.doi.org/https://doi.org/10.1016/j.jcta.2013.09.008}{\path{doi}}.

\bibitem{McNam03}
P.~McNamara, \emph{{$EL$}-labelings, supersolvability and $0$-{H}ecke algebra
  actions on posets}, Journal of Combinatorial Theory, Series A \textbf{101}
  (2003), no.~1, 69--89,
  \href{http://dx.doi.org/https://doi.org/10.1016/S0097-3165(02)00019-5}{\path{doi}}.

\bibitem{Pos06}
A.~Postnikov, \emph{Total positivity, {G}rassmanninas, and networks}, preprint
  (2006), 79 pages,
  \href{http://arxiv.org/abs/math/0609764}{\path{arXiv:math/0609764}}.

\bibitem{ShiZinJus12}
K.~Shigechi and P.~Zinn-Justin, \emph{Path representation of maximal parabolic
  {K}azhdan–{L}usztig polynomials}, J. Pure Appl. Algebra \textbf{216}
  (2012), no.~11, 2533--2548,
  \href{http://dx.doi.org/https://doi.org/10.1016/j.jpaa.2012.03.027}{\path{doi}}.

\bibitem{Sta72}
R.~P. Stanley, \emph{Supersolvable lattices}, Algebra Universalis \textbf{2}
  (1972), 197--217.

\bibitem{Sta74}
\bysame, \emph{Finite lattices and {J}ordan--{H}{\"o}lder sets}, Algebra
  Universalis \textbf{4} (1974), 361--371.

\bibitem{Sta94}
\bysame, \emph{A survey of {E}ulerian posets}, Polytopes: Abstract, Convex, and
  Computational (T.~Bisztriczky, P.~McMullen, R.~Schneider, and A.~I. Weiss,
  eds.), vol. 440, NATO ASI C, Kluwer Academic Publishers,
  Dordrecht/Boston/London, 1994, pp.~301--333.

\bibitem{Sta97b1}
\bysame, \emph{{E}numerative {C}ombinatorics}, vol.~1, Cambridge University
  Press, 1997.

\end{thebibliography}

\end{document}